\documentclass{article}
\usepackage{graphicx} 
\usepackage{amsmath, amsthm, amsfonts, amssymb, amscd, bm, enumerate}
\usepackage{verbatim}
\usepackage{url}

\newtheorem{theorem}{Theorem}
\newtheorem{corollary}{Corollary}
\newtheorem{lemma}{Lemma}
\newtheorem{proposition}{Proposition}
\newtheorem{definition}{Definition}
\newtheorem{assumption}{Assumption}

\theoremstyle{remark}
\newtheorem{remark}{Remark}
\newtheorem{example}{Example}
\def\F{\mathcal{F}}
\def\G{\mathcal{G}}
\def\CE{\mathcal{E}}

\def\CM{\mathcal{M}}

\def\bR{\mathbb{R}}
\def\bN{\mathbb{N}}

\begin{document}

\title{Data-driven nonlinear expectations for statistical uncertainty in decisions}
\author{Samuel N. Cohen\\ Mathematical Institute, University of Oxford\\samuel.cohen@maths.ox.ac.uk}
\date{\today}

\maketitle
\begin{abstract}
In stochastic decision problems, one often wants to estimate the underlying probability measure statistically, and then to use this estimate as a basis for decisions. We shall consider how the uncertainty in this estimation can be explicitly and consistently incorporated in the valuation of decisions, using the theory of nonlinear expectations. 


\end{abstract}

\begin{example}\label{example1}
 Consider the following problem. Let $\{X_n\}_{n\in \bN}$ be identical independent Bernoulli random variables, with unknown parameter $p = P(X_n=1)= 1-P(X_n=0)$, i.e. independent tosses of the same, possibly unfair, coin. You observe $\{X_n\}_{n=1}^{N}$, and then need to draw a conclusion about the likely behaviour of an iid trial $X$. 

In a classical frequentist framework, this is straightforward: the estimator of $p$ (either from MLE or moment matching) is given by $\hat p = S_N/N$, where $S_N = \sum_{n=1}^N X_N$; this estimate has sampling variance $p(1-p)/N\approx \hat p(1-\hat p)/N$.

Suppose we need to evaluate a wager on $X$. Given a loss function $\phi$, we would then usually calculate the expected loss $E[\phi(X)]$, where the expectation is based on the estimated parameters. Without loss of generality, we can assume $\phi(0)=0$, so the inferred expectation is simply given by 
\[\hat E[\phi(X)] = \hat p \phi(1).\]

This leads to a surprising conclusion: the precision of the estimate of $p$ has no impact on our assessment of the wager. To see this, consider a sample based on $N'\gg N$ observations, but with the same value of $\hat p$ . Then the precision of the estimate (as indicated by the reciprocal of the sampling variance) is much higher, but the expected loss of the wager remains identical. Consequently, when considering this wager, this approach concludes that you are indifferent between the settings when $p$ is known precisely or imprecisely. For example, suppose there were two coins, the first was thrown $3$ times with $2$ heads, the second $3000$ times with $2000$ heads. The estimated-expected-loss criterion then states that you are indifferent in choosing which coin to bet on, which is contrary to experience. Note that this conclusion is not changed by the presence of the loss function $\phi$.

Now, some may argue that this is a particular flaw in the frequentist point-estimate approach, as the error of the estimate of $p$ is not part of the probabilistic framework we use when calculating the expectation. So, let's take a Bayesian approach and put a prior on $p$, for example the (conjugate) Beta distribution $B(\alpha, \beta)$. The posterior distribution is then $B(\alpha+S_N, \beta+N-S_N)$; this has mean $\mu_p:=(\alpha+S_N)/(\beta+\alpha+N)$ and variance $\mu_p(1-\mu_p)/(\beta+\alpha+N+1)$.
The posterior expected loss is then $\mu_p \phi(1)$; again this does not depend on the precision of the estimate. 

The choice of prior used is immaterial, as the behaviour is determined by (writing $\F_N$ for the $\sigma$-algebra generated by our observations)
\[E[\phi(X)|\F_N] = E\big[E[\phi(X)|p,\F_N]\big|\F_N\big]=E[p\phi(1)|\F_N] = E[p|\F_N] \phi(1)\]
so only the posterior mean value of $p$ has any impact, not its posterior variance (or any other measure of uncertainty). Even if we extend beyond taking an expected payoff, for example to considering a posterior mean-variance criterion, we would find that the posterior variance of $\phi(X)$ is 
\[E\Big[\Big(\phi(X)-E[\phi(X)|\F_N]\Big)^2\Big|\F_N\Big]= E[p|\F_N](1-E[p|\F_N])\phi(1)^2\]
which still only depends on the posterior mean of $p$. The same conclusion will be reached for any criterion which depends only on the posterior law of $\phi(X)$.

From this, we can conclude both the frequentist and Bayesian expected loss approaches fail to incorporate uncertainty in $p$ in our decision making, in this simple setting\footnote{The mathematical reason for this is that a mixture of Bernoulli random variables is again a Bernoulli random variable. Therefore, at the level of the marginal distribution of $X$, every hierarchical model is equivalent to a non-hierarchical model, and a Bayesian approach adds little mathematically. In other words, as a Bernoulli distribution is a one-parameter family, the posterior distribution can only remember a single value -- the estimated probability -- so there is nowhere to `store' knowledge of the precision of the estimate. The simplicity of this setting may seem contrived, but demonstrates that one cannot, in general, claim that a Bayesian posterior expected loss approach is sufficient to deal with all forms of uncertainty.}.
\end{example}

The unusual behaviour of this type of example has been noticed before. For example, Keynes remarks (using the term `evidential weight' to indicate a concept similar to the precision of probabilities):
\begin{quotation}
 For in deciding on a course of action, it seems plausible to suppose that we ought to take account of the weight as well as the probability of different expectations. ---J.M. Keynes, A Treatise on Probability\footnote{This idea is discussed at length in Keynes' treatise, but is not pursued as a principle in statistics, as is shown by the next sentence: ``But it is difficult to think of any clear example of this, and I do not feel sure that the theory of `evidential weight' has much practical significance.'' In some sense, the aim of this paper is to address this lack of examples in a concrete mathematical fashion, and to propose practical solutions based on classical statistical methods.}, 1921 \cite[p.76]{Keynes1921}
\end{quotation}
Knight \cite{Knight1921} argues that ignoring this uncertainty is not descriptive of people's actions -- we do, generally, have a strict preference for knowledge of the probabilities of outcomes (see also the more general criticism of Allais \cite{Allais1953}). This leads him to distinguish between the concepts of `risk', which is associated with the outcome of $X$ given $p$, and `uncertainty'\footnote{This is a significant simplification of Knight's argument, which also looks at the question of estimating probabilities of future events, which by their very nature, are not the same as events which have already occurred. Nevertheless, the terminology of `Knightian uncertainty' has become common as referring to lack of knowledge of probabilities, so we retain this usage.}, which is associated with our lack of knowledge of $p$.

Within either of the two classical frameworks considered above, there is a natural and classical way to deal with this issue. For a frequentist, instead of using the point estimate $\hat p$, one could consider building a confidence interval for $p$, and then comparing wagers by their worst expectation among parameters within the confidence interval. As the sample size increases, the confidence interval shrinks, and so (for a fixed value of $\hat p$) the value of the wager increases. Similarly for a Bayesian, using a credible interval in the place of the confidence interval. While well known and sensible, this is (at least on the surface) an ad hoc fix, and needs to be defended philosophically: for example, in the Bayesian setting, the uncertainty in $p$ should already have been included in the assessment of $E[\phi(X)|\F_N]$, so this approach seems to be double-counting the uncertainty. In more complex settings, where the parameter $p$ is replaced by a multidimensional parameter and we are interested in comparing the values of a variety of random outcomes (whose expectations are generally nonlinear functions of the parameters), confidence sets become less natural, so a more general and rigorous approach seems to be needed. 

In this paper, we will give one such approach. As Example \ref{example1} shows, to fully incorporate our statistical uncertainty, we cannot simply estimate the (posterior) distribution of the outcome. Instead, we need to retain some knowledge of how accurate that estimate is, and feed that additional knowledge into our decision making.  We shall do this by making a general suggestion of a method, proving some of its general properties, and giving a selection of pertinent examples.

This is our key philosophical claim:
\begin{quotation}When evaluating outcomes in the presence of estimation and model uncertainty, it is not enough to depend simply on the  distribution of the outcome under a fitted model; the evaluation should also depend on \emph{how well other parameter choices and models would have fitted the observations} on which we are basing our evaluation. 
\end{quotation}

Instead of simply dealing with a single probability, we will study the effect of using the likelihood function (which indicates how well a model fits our observations) to generate a `convex expectation', closely related to the risk measures often studied in mathematical finance. The theory of these nonlinear expectations is explored in detail in F\"ollmer and Schied \cite{Follmer2002} (up to some changes of sign), and gives a mathematically rigorous way to deal with `Knightian uncertainty'. In economics, this is closely linked to Gilboa and Schmeidler's model of multiple priors \cite{Gilboa1989}. However, little work has been done on connecting nonlinear expectations with statistics. 

For Example \ref{example1} above, our proposal amounts to the following. Instead of working with the expected loss $E[\phi(X)]$ under one particular estimated measure, consider the quantity
\[\CE(\phi(X)) = \sup_{q\in[0,1]} \big\{q\phi(1) +(1-q)\phi(0) - (k^{-1}\alpha(q))^\gamma\big\}\]
for a fixed uncertainty aversion parameter $k>0$ and exponent $\gamma\geq 1$, where $\alpha$ is the negative log-likelihood of our observations, shifted to have minimal value zero, that is (for $\hat p=S_N/N$ as above),
\begin{equation}\label{eq:coinpenalty}
 \alpha(q) = N\bigg(\hat p \log\Big(\frac{\hat p}{q}\Big) + (1-\hat p) \log\Big(\frac{1-\hat p}{1-q}\Big)\bigg) \approx \frac{N}{\hat p(1-\hat p)}(q-\hat p)^2,
\end{equation}
where the approximation is for large $N$, in a sense to be explored later (it is essentially a form of the central limit theorem, see Section \ref{sec:parametric}). As $\alpha(\hat p)=0$, writing $\xi=\phi(X)$ we have $\CE(\xi) \geq E[\xi]$, and basic calculation gives a (rather inelegant) formula for $\CE(\xi)$ in terms of $\hat p$, $N$, $k$, $\phi(1)$ and $\phi(0)$. The operator $\CE$ gives an `upper' expectation for the loss, depending on the certainty of our parameter estimate given the sample. In effect, we are considering all possible values for $p$, and using our data to determine how reasonable we think they are (as indicated by $-(k^{-1}\alpha(p))^\gamma$). In effect, we do not attempt to give any point-estimate of $p$, or assume that we can treat $p$ as a random variable with known distribution. 

If we were to use $\CE$ to choose between a family of wagers $\phi_i$, we would obtain a classical minimax or `robust optimization' problem (see for example Ben-Tal, El Ghaoui and Nemirovski \cite{Ben-Tal2009}),
 \[\min_i\CE(\phi_i(X))=\min_i\sup_{q\in[0,1]} \big\{q\phi_i(1) +(1-q)\phi_i(0)- (k^{-1}\alpha(q))^\gamma\big\}.\]

The expectation $\CE$ can be thought of as an `upper' expectation, and is convex. The corresponding `lower' expectation $-\CE(-\xi)$ can also be defined, and is concave. This leads naturally to 
\[\Big[-\CE(-\xi), \CE(\xi)\Big]\]
 as an interval prediction for $\xi$. Comparing with more familiar quantities, such as (frequentist) confidence intervals, (Bayesian) credible intervals and upper and lower probabilities in Dempster--Schafer theory, we see that an interval estimate is a natural object to study when describing uncertainty in parameters. We shall see that confidence intervals (in particular, likelihood intervals) arise as a special case of our approach.

\begin{remark}
 The approach taken here is specifically tailored to consider `uncertainty' (lack of knowledge of probabilities), rather than `risk' (lack of knowledge of outcomes, but with known probabilities). In particular, if we have sufficient data that we know the probabiltiy measure exactly, then our expectation is simply the classical expected value, and does not involve any loss-aversion. A loss or utility function can be used to incorporate these effects, or our approach can be extended to allow a wider class of evaluations (see Remark \ref{rem:parametricextension})
\end{remark}

This article proceeds as follows: First, we give a summary of some of the basic properties of nonlinear expectations. Secondly, we consider the effect of using the log-likelihood as the basis for a penalty function and the corresponding ``divergence-robust nonlinear expectations'', and their connection to relative entropy. Using this, we tease out generic large-sample approximations, in both parametric and non-parametric settings. Finally, we consider the connection between divergence-robust expectations and robust statistics (in particular $M$-estimates).

\section{Nonlinear expectations}
 In this section we introduce the concepts of nonlinear expectations and convex risk measures, and discuss their connection with penalty functions on the space of measures. These objects provide a technical foundation with which to model the presence of uncertainty in a random setting. This theory is explored in some detail in F\"ollmer and Schied \cite{Follmer2002} and Frittelli and Rosazza-Gianin \cite{Frittelli2002}, among many others. We here present, without proof, the key details of this theory as needed for our analysis. 
 
 \begin{definition}\label{defn:nonlinearexpectation}
  Let $(\Omega, \F, P)$ be a probability space, and $L^\infty(\F)$ denote the space of $P$-essentially bounded $\F$-measurable random variables. A nonlinear expectation on $L^\infty(\F)$ is a mapping 
 \[\CE:L^\infty(\F) \to \bR\]
 satisfying the assumptions, 
 \begin{itemize}
  \item Strict Monotonicity: for any $\xi_1, \xi_2\in L^\infty(\F)$, if $\xi_1\geq \xi_2$ a.s. then $\CE(\xi_1) \geq \CE(\xi_2)$, and if in addition $\CE(\xi_1)=\CE(\xi_2)$ then $\xi_1=\xi_2$ a.s.
  \item Constant triviality: for any constant $k\in \bR$, $\CE(k)=k$.
  \item Translation equivariance: for any $k\in\bR$, $\xi\in L^\infty(\F)$, $\CE(\xi+k)= \CE(\xi)+k$.
 \end{itemize}
 A `convex' expectation in addition satisfies
 \begin{itemize}
  \item Convexity: for any $\lambda\in [0,1]$, $\xi_1, \xi_2\in L^\infty(\F)$, 
 \[\CE(\lambda \xi_1+ (1-\lambda) \xi_2) \leq \lambda \CE(\xi_1)+ (1-\lambda) \CE(\xi_2).\]
 \end{itemize}
 
 If $\CE$ is a convex expectation, then the operator defined by $\rho(\xi) = \CE(-\xi)$ is called a \emph{convex risk measure}. A particularly nice class of convex expectations is those which satisfy 
 \begin{itemize}
  \item Lower semicontinuity: For a sequence $\{\xi_n \}_{n\in\bN}\subset L^\infty(\F)$ with $\xi_n \uparrow \xi \in L^\infty(\F)$ pointwise, $\CE(\xi_n) \uparrow \CE(\xi)$.
 \end{itemize}
 \end{definition}
 
 The following theorem (which was expressed in the language of risk measures) is due to F\"ollmer and Scheid \cite{Follmer2002} and Frittelli and Rosazza-Gianin \cite{Frittelli2002}.
 \begin{theorem}\label{thm:penaltyexists}
  Let $\CM_1$ denote the space of all probability measures on $(\Omega, \F)$ absolutely continuous with respect to $P$. Suppose $\CE$ is a lower semicontinuous convex expectation. Then there exists a  `penalty' function $\alpha: \CM_1\to [0,\infty]$ such that 
 \[\CE(\xi) = \sup_{Q\in \CM_1} \big\{E_Q[\xi]-\alpha(Q)\big\}.\]
 In addition, there is a minimal such function, given by 
 \[\alpha_{\min}(Q) = \sup_{\xi\in L^\infty(\G)} \big\{E_Q[\xi]-\CE(\xi)\big\}\qquad \text{ for }Q\in \CM_1.\]
Provided $\alpha(Q)<\infty$ for some $Q$ equivalent to $P$, we can restrict our attention to measures in $\CM_1$ equivalent to $P$ without loss of generality.
 \end{theorem}
 
 Classic convex analysis shows that $\alpha_{\min}$ is the Fenchel--Legendre conjugate of $\CE$, and is also a convex function and weak-* lower semicontinuous. It is also clear that, as $\CE(0)=0$,  we have the identity $\inf_{Q\in \CM_1}\{\alpha_{\min}(Q)\} =0$.
 
 \begin{remark}
  As discussed above in the context of our example, this result gives some intuition as to how a convex expectation can model `Knightian' uncertainty. One considers all the possible probability measures on the space, and then selects the maximal expectation among all measures, penalizing each measure depending on how plausible it is considered. As convexity of $\CE$ is a natural requirement of an `uncertainty averse' assessment of outcomes, Theorem \ref{thm:penaltyexists} shows that this is the only way to construct an `expectation' $\CE$ which penalizes uncertainty, while preserving monotonicity, translation equivariance and constant triviality.

 In particular, if (and only if) $\CE$ is positively homogenous, that is, it satisfies ``for any $\lambda\geq 0$, $\CE(\lambda \xi) = \lambda \CE(\xi)$'',  then $\alpha_{\min}$ only takes the values $\{0, \infty\}$, and we can rewrite our representation as
 \[\CE(\xi) = \sup_{Q\in \CM^*} \big\{E_Q[\xi]\big\}.\]
 where $\CM^*\subset \CM_1$ is the set of measures for which $\alpha(Q)= 0$. In this case, we see that our convex expectation corresponds to taking the maximum of the expectations under a range of possible models for the random system.
 \end{remark}

\begin{remark}
 The convex expectation $\CE$ is defined above as an operator on $L^\infty$. However, given the equivalent representation 
 \[\CE(\xi) = \sup_{Q\in \CM_1: \alpha(Q)<\infty} \big\{E_Q[\xi] -\alpha(Q)\big\},\]
we can clearly define $\CE(\xi)$ for a wider class of random variables. In particular, $\CE(\xi)$ is well defined (but may be infinite) for all random variables $\xi$ such that $E_Q[\xi]>-\infty$  for every $Q\in \CM_1$ with $\alpha(Q)<\infty$.
\end{remark}

Given a convex nonlinear expectation $\CE$, there is a natural class of `acceptable' random variables for a decision problem, namely (given we evaluate losses) the convex level set 
\[\mathcal{A}=\{\xi: \CE(\xi)\leq 0\}.\]
One can also use a nonlinear expectation as a value to be optimized; in this setting the convexity of the operator is of significant interest. Finally, one can use a nonlinear expectation to give a robust point estimate of $\xi$, given a loss function $\phi$, by choosing the value $\hat\xi\in\bR$ which minimizes the loss $\CE(\phi(\xi-\hat\xi))$ (cf. Wald \cite{Wald1945}).

\section{Penalties and likelihood}
The general framework of nonlinear expectations is well suited to modelling Knightian uncertainty, but is not usually connected with statistical estimation. We would like to have a general principle for treating our uncertainty, which is closely tied to classical statistics. Our aim is to have a nonlinear expectation \emph{which uses observations to derive estimates of real-world probabilities, and uses these estimates and their uncertainty to give robust average values for a wide range of random outcomes}. Rather than continuing to take an abstract axiomatic approach, we shall consider the following concrete proposal:

\begin{definition}
Suppose we have an observation vector $\mathbf{x}$ taking values in $\bR^N$. For a model $Q\in \CM_1$, let $L(Q|\mathbf{x})$ denote the likelihood of $\mathbf{x}$ under $Q$, that is the density of $\mathbf{x}$ with respect to a reference measure (which we shall take to be Lebesgue measure on $\bR^N$ for simplicity).

Let $\mathcal{Q}\subseteq\CM_1$ be a set of models under consideration (for example, a parametric set of distributions).  We then define the ``$\mathcal{Q}|\mathbf{x}$-divergence'' to be the negative log-likelihood ratio
\[\alpha_{\mathcal{Q}|\mathbf{x}}(Q):= -\log\big(L(Q|\mathbf{x})\big) + \sup_{\tilde Q\in \mathcal{Q}}\Big\{\log\big(L(\tilde Q|\mathbf{x})\big)\Big\}.\]
The right hand side is well defined whether or not a maximum likelihood estimator\footnote{Recall that a $\mathcal{Q}$-MLE (maximum likelihood estimator) is a measurable map $\mathbf{x}\to \hat Q\in \mathcal{Q}$ such that $L(\hat Q|\mathbf{x}) \geq L(Q|\mathbf{x})$ for all $Q\in\mathcal{Q}$. We say that a quantity $Y$ is a $\mathcal{Q}$-MLE for $E_Q[\xi]$ if $Y=E_{\hat Q}[\xi]$ where $\hat Q$ is a $\mathcal{Q}$-MLE.} exists. Given a $\mathcal{Q}$-MLE $\hat Q$, we would have the simpler representation 
\[\alpha_{\mathcal{Q}|\mathbf{x}}(Q):= -\log\Big(\frac{L(Q|\mathbf{x})}{L(\hat Q|\mathbf{x})}\Big).\]

Given $\alpha_{\mathcal{Q}|\mathbf{x}}$, for an uncertainty aversion parameter $k>0$ and exponent $\gamma\in[1,\infty]$, we obtain the corresponding convex expectation
\[\mathcal{E}_{\mathcal{Q}|\mathbf{x}}^{k,\gamma}(\xi):= \sup_{Q\in \mathcal{Q}}\Big\{E_Q[\xi|\mathbf{x}] -\Big(\frac{1}{k}\alpha_{\mathcal{Q}|\mathbf{x}}(Q)\Big)^\gamma\Big\}\]
where we adopt the convention $x^\infty = 0$ for $x\in[0,1]$ and $+\infty$ otherwise. We call $\mathcal{E}_{\mathcal{Q}|\mathbf{x}}^{k,\gamma}$ the  ``$\mathcal{Q}|\mathbf{x}$-divergence robust expectation" (with parameter $k,\gamma$), or simply the ``DR-expectation\footnote{This acronym could also stand for `Data-driven Robust expectation', which may be a preferable emphasis.}''. 
\end{definition}

Our attention will mainly be on the two extremal cases $\gamma=1$ and $\gamma=\infty$, however the intervening cases are natural interpolations between them. The statement $1^\infty=0$ is natural from a convex analytic perspective, as it implies $|x|^q$ is proportional to the convex dual of $|x|^p$, whenever $p^{-1}+q^{-1}=1$, for $p\in[1,\infty]$.

We shall focus our attention on the special case where $\mathbf{x}=\{X_n\}_{n=1}^N$ and, under each $Q\in\mathcal{Q}$, we know $X,\{X_n\}_{n\in\bN}$ are iid random variables -- this allows analytically tractable results, however our approach is applicable much more widely.

\begin{remark}
In our example above, $\mathcal{Q}$ corresponds to the set of measures such that $X,\{X_n\}_{n=1}^{N}$ are iid Bernoulli with parameter $p\in[0,1]$. In this example, we did not consider all measures in $\CM_1$ (this would include, for example, models where $\{X_n\}_{n=1}^N$ and $X$ come from completely unrelated distributions), but neither did we restrict our attention to a single $Q\in\mathcal{Q}$.
\end{remark}

Typically the operator $\CE$ cannot be evaluated by hand, instead numerical optimization or approximation is needed. In the setting of Example \ref{example1} above, if $\gamma=1$ then a closed form representation can be obtained, however is quite inelegant (the optimal $q$ is the solution to a quadratic equation, but the resulting equation for $\CE_{\mathcal{Q}|\mathbf{x}}^{k,\gamma}(\xi)$ does not simplify). A simple example where closed form quantities can be derived is the classic setting where the data are assumed to be Gaussian with unknown mean (and as alluded to above, these are unsurprisingly very similar for large $N$).
\begin{example}\label{exampleNormal}
Suppose $\mathbf{x}=(X_1, X_2,...,X_N)$ and $\mathcal{Q}$ corresponds to those measures under which $X,\{X_n\}_{n=1}^N$ are iid $N(\mu, 1)$ random variables, where $\mu$ is unknown. Then, if $\bar X=N^{-1}\sum_{n=1}^N X_n$ denotes the sample mean, for any constant $\beta>0$, simple calculus can be used to derive
\[\begin{split}
   \CE_{\mathcal{Q}|\mathbf{x}}^{k,\gamma}(\beta X) &= \sup_{\mu\in \bR} \Big\{\beta\mu - \Big(\frac{1}{2k}\Big(\sum_{n=1}^N (\mu-X_n)^2-\sum_{n=1}^N(\bar X-X_n)^2\Big)\Big)^\gamma\Big\}\\
&= \sup_{\mu\in \bR} \Big\{\beta\mu - \Big(\frac{N}{2k}(\mu-\bar X )^2\Big)^\gamma\Big\}\\
&= \beta \bar X + \beta^{\frac{2\gamma}{2\gamma-1}}\Big(\frac{k}{N}\Big)^{\frac{\gamma}{2\gamma-1}}(2\gamma)^{\frac{-1}{2\gamma-1}}\Big(1-\frac{1}{2\gamma}\Big).
  \end{split}
\]
In particular, when $\gamma=1$, we have
\[ \CE_{\mathcal{Q}|\mathbf{x}}^{k,1}(\beta X) = \beta\bar X +\frac{\beta^2 k}{2N} = \beta \bar X + \frac{k}{2}\, \mathrm{Var}(\beta \bar X)\]
and, taking the limit $\gamma\to\infty$ (or directly from the definition),
\[ \CE_{\mathcal{Q}|\mathbf{x}}^{k,\infty}(\beta X) = \beta\bar X + \beta \sqrt{\frac{2k}{N}}=\beta\bar X +  \sqrt{2k}\,\mathrm{Sd}(\beta\bar X).\]
In this latter case, taking $k\approx 2$, we obtain the upper bound of the classical $95\%$ confidence interval for $\beta X$.
 
The corresponding lower expectations are given by the symmetric quantities
\[-\CE_{\mathcal{Q}|\mathbf{x}}^{k,1}(-\beta X)=\beta\bar X -\frac{\beta^2 k}{2N}, \qquad -\CE_{\mathcal{Q}|\mathbf{x}}^{k,\infty}(-\beta X) = \beta\bar X - \beta \sqrt{\frac{2k}{N}}.\] From this example we can observe a few phenomena, which we will discuss more generally below. First, for $\gamma<\infty$, $\CE_{\mathcal{Q}|\mathbf{x}}^{k,\gamma}(\beta X)$ is not positively homogenous in $\beta$, that is, $\CE_{\mathcal{Q}|\mathbf{x}}^{k,\gamma}(|\beta| X)\neq |\beta|\CE_{\mathcal{Q}|\mathbf{x}}^{k,\gamma}(X)$. The larger (in absolute terms) the random variable considered, the more the uncertainty affects our assessment. On the other hand, this requirement is satisfied when $\gamma=\infty$, and there is a close relationship between $\CE^{k,\infty}_{\mathcal{Q}|\mathbf{x}}$ and the classical confidence interval for $E[X]$.

Secondly, for any $\gamma$, as the ratio of the uncertainty parameter and the sample size $k/N \to 0$, the DR-expectation converges to the (unique)  $\mathcal{Q}$-MLE $\beta \bar X$ (that is, the parameter corresponding to the measure in $\mathcal{Q}$ with the largest likelihood). This convergence is of the order $(k/N)^{\frac{\gamma}{2\gamma-1}}$.

In this setting we can also calculate, for $\beta>0$,
\[\begin{split} \CE_{\mathcal{Q}|\mathbf{x}}^{k,1}(\beta X^2) &= \sup_{\mu\in \bR} \Big\{\beta(1+\mu^2) - \frac{N}{2k}(\mu-\bar X )^2\Big\}\\
&=\begin{cases}
\beta + \beta \Big(\frac{N(N-2k\beta^2)}{(N-2k\beta)^2}\Big)\bar X^2 & \beta< N/2k\\
+\infty & \beta \geq N/2k
\end{cases}
     \end{split}
\]
whereas 
\[\CE_{\mathcal{Q}|\mathbf{x}}^{k,\infty}(\beta X^2) = \beta\Big(1+\bar X + \sqrt{{2k}/{N}}\Big)^2,\] 
which is always finite. This explosion in $\CE^{k,1}_{\mathcal{Q}|\mathbf{x}}$ will be considered in more detail in Section \ref{sec:robust}. Notice that again, as $k/N\to 0$, \[\CE_{\mathcal{Q}|\mathbf{x}}^{k, \gamma}(\beta X^2)\to \beta(1+\bar X^2),\] which is the  $\mathcal{Q}$-MLE for $E[\beta X^2]$.
\end{example}

\begin{remark}
 We will focus on the use of the likelihood for estimation, however it is clear that other quantities could also be considered. In particular, if we have a family of parametric distributions $\mathcal{Q}$ with varying numbers of parameters, then it would be reasonable to penalize by Akaike's information criterion or the Bayesian information criterion, rather than simply by the likelihood. In some settings, the use of a quasi-likelihood, rather than the true likelihood, may also be of interest, particularly if this renders the problem more computationally efficient. Furthermore, including terms relating to the log-density of a `prior' penalty may be of interest, so the penalty will be taken using the log-density of the posterior distribution, rather than with the likelihood (see Example \ref{meanvariancexample} for this). For the sake of simplicity, we will not pursue these variants in detail here, however their behaviour should be qualitatively similar to what we consider. 
\end{remark}

We have noticed above, in the Gaussian case, that our nonlinear expectation is positively homogeneous only in the case $\gamma=\infty$. This is a general fact, as shown by the following proposition.
\begin{proposition}
In the case $\gamma=\infty$ (and only in this case, provided the likelihood is finite and varying for a nontrivial subset of $\mathcal{Q}$), our nonlinear expectation is positively homogeneous, that is 
\[\CE_{\mathcal{Q}|\mathbf{x}}^{k,\gamma}(\beta \xi) = \beta\CE_{\mathcal{Q}|\mathbf{x}}^{k,\gamma}(\xi) \text{ for all }\beta>0\text{ iff }\gamma=\infty.\]
\end{proposition}
\begin{proof}
It is classical (see for example F\"ollmer and Schied \cite{Follmer2002}) that a convex nonlinear expectation is positively homogeneous if and only if the penalty takes only the values $\{0,\infty\}$. Given the likelihood is finite and varying on a nontrivial subset of $\mathcal{Q}$, this is not the case for any $\gamma<\infty$, but is the case for $\gamma=\infty$ by definition.
\end{proof}

\begin{remark}
For a $\mathcal{Q}$-MLE $\hat Q$, by definition we have  $L(Q|\mathbf{x})\leq L(\hat Q|\mathbf{x})$, so $\alpha_{\mathcal{Q}|\mathbf{x}}(Q)\geq 0$, with equality only if $Q$ is a maximum likelihood model. In general we cannot say whether $\alpha_{\mathcal{Q}|\mathbf{x}}$ is convex (indeed, we have not assumed that its domain $\mathcal{Q}$ is a convex set), so it is not generally the case that $(k^{-1}\alpha_{\mathcal{Q}|\mathbf{x}})^\gamma$ is the minimal penalty for $\CE_{\mathcal{Q}|\mathbf{x}}^{k, \gamma}$.
\end{remark}

\subsection{Dynamic consistency}
Within the theory of nonlinear expectations, much attention has been paid to questions of dynamic consistency. If we have a family $\{\CE_s\}_{s\ge 0}$ of `conditional' nonlinear expectations relative to a filtration $\{\F_t\}_{t\ge 0}$, then dynamic consistency requires, for every $\xi$ and all $s\leq t$, that we have (i) the recursivity relationship $\CE_s(\CE_t(\xi))=\CE_s(\xi)$ and (ii) the relevance condition $\CE_s(I_A \xi) = I_A \CE_s(\xi)$ for all $A\in \F_t$.  This concept is generally not appropriate for our approach, as the expectations we define are typically not consistent. This can be seen from the following easy extension of Example \ref{exampleNormal}.

\begin{example}
 In the context of Example \ref{exampleNormal}, write $\mathbf{x}_N = \{X_1,...,X_N\}$, so $\F_N = \sigma(\mathbf{x}_N)$. We have
\[\begin{split}\CE_{\mathcal{Q}|\mathbf{x}_1}^{k,1}(\CE_{\mathcal{Q}|\mathbf{x}_2}^{k,1}(X)) &= \CE_{\mathcal{Q}|\mathbf{x}_1}^{k,1}\Big(\frac{X_1+X_2}{2} + \frac{k}{4}\Big)= \frac{X_1}{2} + \frac{k}{4} + \CE_{\mathcal{Q}|\mathbf{x}_1}^{k,1}\Big(\frac{X_2}{2}\Big)\\
   &= \frac{X_1}{2} + \frac{k}{4} + \frac{X_1}{2} + \frac{k}{8}= X_1 + \frac{3k}{8}\\
&\neq X_1 + \frac{k}{2} = \CE_{\mathcal{Q}|\mathbf{x}_1}^{k,1}(X)
  \end{split}
\]
and
\[\begin{split}\CE_{\mathcal{Q}|\mathbf{x}_1}^{k,\infty}(\CE_{\mathcal{Q}|\mathbf{x}_2}^{k,\infty}(X)) &= \CE_{\mathcal{Q}|\mathbf{x}_1}^{k,\infty}\Big(\frac{X_1+X_2}{2} + \sqrt{\frac{2k}{2}}\Big)= \frac{X_1}{2} + \sqrt{k} + \CE_{\mathcal{Q}|\mathbf{x}_1}^{k,\infty}\Big(\frac{X_2}{2}\Big)\\
   &= \frac{X_1}{2} + \sqrt{k} + \frac{X_1}{2} + \sqrt{\frac{k}{2}}= X_1 + (1+2^{-1/2})\sqrt{k}\\
&\neq X_1 + \sqrt{2k} = \CE_{\mathcal{Q}|\mathbf{x}_1}^{k,\infty}(X).
  \end{split}
\]
So in either case, the nonlinear expectation $\{\CE_{\mathcal{Q}|\mathbf{x}_N}^{k,\gamma}\}_{N\in\bN}$ is not recursive\footnote{It is curious that, in this setting, recursion can make the nonlinear expectation \emph{less} risk averse, that is, $\CE_{\mathcal{Q}|\mathbf{x}_1}^{k,1}(\CE_{\mathcal{Q}|\mathbf{x}_2}^{k,1}(X))<\CE_{\mathcal{Q}|\mathbf{x}_1}^{k,1}(X)$. This is contrary to the typical behaviour of nonlinear expectations (see, for example, the discussion of scaling in Madan, Pistorius and Stadje \cite{Madan2016}). For $\gamma=\infty$, we have the usual relationship $\CE_{\mathcal{Q}|\mathbf{x}_1}^{k,1}(\CE_{\mathcal{Q}|\mathbf{x}_2}^{k,1}(X))>\CE_{\mathcal{Q}|\mathbf{x}_1}^{k,1}(X)$, but for other values of $\gamma$ (e.g. $\gamma=2$), one can observe a non-monotone relationship between the number of recursion steps and the level of risk aversion.}.
\end{example}

In effect, our problem differs from the dynamically consistent one in the following (closely related) key ways:
\begin{itemize}
\item In a dynamically consistent setting, the penalty is prescribed (and may be taken to be constant through time, as discussed in \cite{Cohen2015b}) while the observations lead to conditional expectations appearing in the nonlinear expectation. In our setting, the penalty is \emph{determined by the observations} (through the $\mathcal{Q}|\mathbf{x}$-divergence), and so only the family of models $\mathcal{Q}$ and the constants $k,\gamma$ need to be specified. In this way, the observations will inform our understanding of the real-world probabilities directly, rather than simply replacing them with conditional probabilties.
\item In a dynamically consistent setting, decisions are typically thought of as being made at each time point, and one needs to ensure that they satisfy a dynamic programming principle (so it is reasonable to make plans for future decisions). In our setting, we naturally consider making a single decision after repeated observations, and seek an empirical basis on which to do so.
\item In a dynamically consistent setting, one typically has that the limit as the number of observations increases is the true value of the outcome, that is $\CE(\xi|\F_t)\to \xi$ as $t\to T$. In our setting, we will instead typically have that $\CE_{\mathcal{Q}|\mathbf{x}_N}^{k,\gamma}(\phi(X)) \to E_P[\phi(X)]$ as $N\to \infty$ in $P$-probability for all $P\in\mathcal{Q}$, that is, our expectation converges to the expectation under the `true' measure.
\item In a dynamically consistent setting, if we assume our observations are independent of $X$ (under all $Q\in\mathcal{Q}$), we will typically learn nothing about the value of $X$, so $\CE(\phi(X)|\F_N)$ is constant. In our setting, independent observations are needed to teach us about the distribution of $X$, and hence give useful information, which is incorporated into our expectations.
\item In a dynamically consistent setting, the underlying models are typically required to be stable under pasting through time. Conceptually, this implies that there is no significant link assumed between the `true' model governing our observations at different times\footnote{In particular, new observations cannot affect our opinions on the measure which was active at an earlier time. As previously mentioned, this connects more generally with the concerns of Knight \cite{Knight1921}, who discusses the problem that observations at different times may be from different models. The difficulty lies in the fact that, without some presumption of homogeneity in nature, statistical inference is impossible.}. Conversely, in our setting, we typically assume that the underlying model is constant through time (i.e. our observations are iid), and hence repeated observations \emph{can} inform our view of the `true' model.
\end{itemize}
In \cite{Cohen2015b}, a filtering problem was considered, where $\xi$ was taken to depend on a hidden time-varying process, and the observations were used to filter the expected value of $\xi$. This was coupled with a nonlinear expectation in two ways. First, a DR-expectation approach was used in an initial calibration phase to estimate the underlying probabilistic structures of the model and their uncertainty, that is, the dynamics of the hidden process and the observation process. Secondly, the penalty function obtained from the DR-expectation approach was used to build a nonlinear expectation with good dynamic properties for new realizations of these processes, associated with an on-line filter. In this second setting, new information is incorporated into the risk assessment, but \emph{one does not recalibrate the estimation of the underlying probabilistic dynamics}.

Nevertheless, as we shall see, some `dynamic' properties of our nonlinear expectation are available. In particular, in Section \ref{sec:largesample} we shall consider the large-sample asymptotic  behaviour  of a DR-expectation.

\subsection{Exponentials and Entropy}

It will not be a surprise that there is a connection between the convex expectation we propose and a more traditional quantity in risk-averse decision making, namely the certainty equivalent under exponential utility.

\begin{definition}
 For a random variable $\xi$, under a reference measure $P$, the certainty equivalent under exponential utility has definition
\[\CE_{\mathrm{exp}}^k(\xi) = \frac{1}{k} \log E_P[\exp(k\xi)]\]
where $k>0$ is a risk-aversion parameter. Defining the relative entropy (or Kullback--Liebler divergence) 
\[D_{\mathrm{KL}}(Q||P) = E_Q\Big[\log\Big(\frac{dQ}{dP}\Big)\Big] = E_P\Big[\frac{dQ}{dP}\log\Big(\frac{dQ}{dP}\Big)\Big]\]
we have the representation (see, for example, \cite{Follmer2002})
\[\CE_{\mathrm{exp}}^k(\xi) = \sup_{Q\in \CM_1}\Big\{E_Q[\xi] - \frac{1}{k} D_{\mathrm{KL}}(Q||P)\Big\}.\]
\end{definition}
Replacing expectations by conditional expectations, we obtain the conditional certainty equivalent.

\begin{remark}
It is useful to consider the relative entropy of the law of $X$ separately from the other observations $\{X_n\}_{n\in\bN}$. We therefore define
\[D_{\mathrm{KL}|X}(Q||P) =\int \log\Big(\frac{f(x,Q)}{f(x,P)}\Big)f(x,Q)dx.\]
Assuming $\hat Q\approx P$, where $P$ is the `real world' probability measure, in light of the law of large numbers we hope for a simple connection, at least asymptotically, between the scaled deviance
\[\frac{1}{N}\alpha_{\mathcal{Q}|\mathbf{x}}(Q) = -\frac{1}{N}\sum_{n=1}^N \log \Big(\frac{f(X_n, Q)}{f(X_n, \hat Q)}\Big)\approx -E_P\Big[ \log \Big(\frac{f(X, Q)}{f(X, \hat Q)}\Big)\Big]\approx D_{\mathrm{KL}|X}(P||Q)\]
 and the penalty in the exponential utility, that is, $D_{\mathrm{KL}}(Q||P)$. In general, this is made more difficult by the fact we have an infinite family of measures $\mathcal{Q}$, and by the lack of symmetry in the relative entropy, as $D_{\mathrm{KL}}(Q||P)\neq D_{\mathrm{KL}}(P||Q)$. We shall pursue this connection in the coming section.
\end{remark}

\begin{remark}
Consider the case of an uncertain symmetric location parameter, that is when $\mathcal{Q}$ is parameterized by an unknown quantity $\theta$, and under each $Q^\theta$, $\{X_n\}_{n=1}^N$ are iid with density $f(|x-\theta|)$. Then one can show that
\[D_{\mathrm{KL}|X}(Q^\theta||Q^{\theta'}) = D_{\mathrm{KL}|X}(Q^{\theta'}||Q^\theta).\] Therefore, in these cases, one could try and equate the nonlinear expectation and the exponential certainty equivalent under the MLE measure $\hat Q$. However, this requires that the measure $Q\in \CM_1$ which maximizes
\[E_Q[X]-D_{\mathrm{KL}|X}(Q||\hat Q)\]
is within the class $\mathcal{Q}$. That is, the optimizer differs from $\hat Q$ only through a change of the parameter $\theta$. In general, this is only the case when the models considered are Gaussian with uncertain mean. Then, with $P=\hat Q$, we have (cf. Example \ref{exampleNormal})
\[\CE^k_{\mathrm{exp}}(\beta X) = \beta \bar X + \frac{\beta^2 k}{2} \mathrm{Var}(X).\]
\end{remark}

\begin{remark}\label{rem:parametricextension}
One extension of our approach is to change the penalty function to include a entropy term taken in the `other' direction, that is, to use the penalty
\[\alpha^{k,\beta}(Q) = \inf_{Q'\in \mathcal{Q}}\Big\{\frac{1}{\beta} D_{\mathrm{KL}}(Q||Q') + \frac{1}{k} \alpha_{\mathcal{Q}|\mathbf{x}}(Q')\Big\}\]
for some $\beta>0$. This is particularly of interest where $\mathcal{Q}$ is a parametric family (and we think our model should be similar, if not precisely the same, as a parametric model) or where we wish to include risk aversion (as measured using exponential utility\footnote{Replacing entropy with a different penalty would allow for other utility functionals to be considered, if desired.}). This is well defined for all measures $Q \in \CM_1$, and gives the expectation:
\[\sup_{Q\in \CM_1}\big\{E_Q[\xi] -\alpha^{k,\beta}(Q)\big\} = \frac{1}{\beta} \log \CE_{\mathcal{Q}|\mathbf{x}}^{k,1}(e^{\beta \xi}).\]
\end{remark}

\section{Large-sample theory}\label{sec:largesample}

In this section, we shall seek to study the large-sample theory of the nonlinear expectation $\CE^{k,\gamma}_{\mathcal{Q}|\mathbf{x}}$. In practice, this is particularly useful to give approximations and qualitative descriptions of its behaviour.

Throughout this section, we shall assume that we have observations $\{X_n\}_{n\in\bN}$, and a family of measures $\mathcal{Q}$ under which $X, \{X_n\}_{n\in\bN}$ are iid random variables with corresponding densities $f(x; Q)dx$. We write $\mathbf{x}_N=(X_1,...,X_N)$. We shall be interested in determining the behaviour, for large $N$, of $\CE_{\mathcal{Q}|\mathbf{x}_N}^{k,\gamma}(\phi(X))$, where $\phi$ is a bounded function. For simplicity, we shall assume that the MLE exists (however our results can be extended to remove this assumption, with an increase in notational complexity). We write $\hat Q_N$ for the $\mathcal{Q}$-MLE based on observations $\mathbf{x}_N$.

 Given the lack of positive homogeneity, it is interesting to consider the behaviour of $c^{-1}_N \CE_{\mathcal{Q}|\mathbf{x}}^{k,\gamma}(c_N\xi)$, where $c_N$ has prescribed growth in $N$. The following lemma allows us to instead vary the uncertainty parameter $k$.

\begin{lemma}
For any $k>0$, any $\gamma<\infty$, any random variable $\xi$,
\[k^{-\gamma}\CE_{\mathcal{Q}|\mathbf{x}}^{1,\gamma}(k^\gamma\xi) = \CE_{\mathcal{Q}|\mathbf{x}}^{k,\gamma}(\xi).\]
\end{lemma}
\begin{proof}
\[\begin{split}k^{-\gamma}\CE_{\mathcal{Q}|\mathbf{x}}^{1,\gamma}(k^\gamma \xi) &= k^{-\gamma}\sup_{Q\in\mathcal{Q}}\Big\{E_Q[k^\gamma \xi|\mathbf{x}] -\big(\alpha_{\mathcal{Q}|\mathbf{x}}(Q)\big)^\gamma\Big\}\\
&= \sup_{Q\in\mathcal{Q}}\Big\{E_Q[ \xi|\mathbf{x}] -\Big(\frac{1}{k}\alpha_{\mathcal{Q}|\mathbf{x}}(Q)\Big)^\gamma\Big\}= \CE_{\mathcal{Q}|\mathbf{x}}^{k,\gamma}(\xi)
\end{split}\]
\end{proof}

To enable a simple description of our asymptotic results, we recall the following definition
\begin{definition}
 For functions $f$ and $g$, we shall shall write $f=O_P(g(N))$ whenever $f(N)/g(N)$ is stochastically bounded (that is, $P(|f(N)/g(N)|>M)\to 0$ as $M\to \infty$ for each $N$) and $f=o_P(g(N))$ whenever $P\text{-}\lim_{N\to\infty}|f(N)/g(N)|=0$. Note that this depends on the choice of measure $P$. The subscript $P$ is omitted in the classical case (i.e. when the convergence is not in probability).
\end{definition}

\subsection{Nonparametric results}
We now give some results when we do not assume $\mathcal{Q}$ comes from a `nice' parametric family. Given we will take a supremum over a family of densities, we need a uniform version of the law of large numbers.
For this reason, we make the following definition.

\begin{definition}
 We say a family $\mathcal{Q}$ is a Glivenko--Cantelli--Donsker class of measures (or GCD class) if, for any $P\in\mathcal{Q}$, 
\[\sup^*_{Q\in \mathcal{Q}}\Big\{\frac{\alpha_{\mathcal{Q}|\mathbf{x}_N}(Q)}{N} - D_{\mathrm{KL}|X}(P||Q)\Big\}=O_P(N^{-1/2})\]
(where $\sup^*$ refers to the minimal measurable envelope, to ensure measurability).
\end{definition}
\begin{remark}
The reason for the name (Glivenko--Cantelli--Donsker) is simply because, if we have a uniform weak Glivenko--Cantelli theorem when indexing the empirical distribution by the family of log-likelihoods, then the term in brackets converges in probability to $0$. If we also have a uniform Donkser theorem, then we know that $\sqrt{N}\big(\alpha_{\mathcal{Q}|\mathbf{x}_N}(Q)/N - D_{\mathrm{KL}}(P||Q)\big)$ converges (in some sense) to a finite-valued Gaussian process, which implies it is of the order stated.

It is easy to show, given consistency of the MLE and some integrability, that a finite family $\mathcal{Q}$ is always a GCD class.
\end{remark}

\begin{lemma}\label{lem:GCDconditions}
 Suppose $\mathcal{Q}$ is a family of measures such that $\{X_n\}_{n\in\bN}$ are iid with respective densities $\{f(\cdot; Q)\}_{Q\in\mathcal{Q}}$ which satisfy
\begin{enumerate}[i)]
 \item there is a compact set $K$, such that for every $P\in \mathcal{Q}$, $P(X\in K)=1$
 \item there is $\epsilon>0$ such that $\epsilon <\inf_{Q\in\mathcal{Q}}\min_{x\in K}f(x, Q)$,
 \item there is $C<\infty$ and $\rho>1/2$ such that, for all $P,Q\in\mathcal{Q}$, the likelihood ratios $f(\cdot, Q)/f(\cdot, P)$ take values in $[C^{-1},C]$ and are uniformly $\rho$-H\"older continuous with norm $C$, that is, writing $L(x) = f(x, Q)/f(x, P)$,
\[\sup_{x,y} \frac{|L(x)-L(y)|}{|x-y|^\rho}\leq C.\]
\end{enumerate}
 Then $\mathcal{Q}$ is a GCD class of measures.
\end{lemma}
\begin{proof}
 See Appendix.
\end{proof}

We can now prove the following versions of the law of large numbers and the central limit theorem. We begin with the case $\gamma=1$.
\begin{theorem}\label{thm:lln1case}
Suppose $\mathcal{Q}$ is a GCD class of measures and $k_N=o(N)$. Consider a random variable $\xi=\phi(X)$, where $\phi$ is a bounded measurable function and $X, \{X_n\}_{n\in \mathbb{N}}$ are iid under every $Q\in\mathcal{Q}$.
\begin{enumerate}[(i)]
\item $\CE^{k_N,1}_{\mathcal{Q}|\mathbf{x}_N}$ is a consistent estimator, that is
\[\CE_{\mathcal{Q}|\mathbf{x}_N}^{k_N,1}(\xi) \to_P E_P[\xi]\] as $N\to\infty$ for every $P\in \mathcal{Q}$.
\item We have the asymptotic behaviour (as $N\to\infty$, for each $P\in \mathcal{Q}$)
\[\CE_{\mathcal{Q}|\mathbf{x}_N}^{k_N,1}(\xi) \leq E_P[\xi]+\frac{k_N }{N}\frac{1}{2}\mathrm{Var}_P(\xi) + O\Big(\Big(\frac{k_N}{N}\Big)^2\Big)+O_P\Big(\frac{N^{1/2}}{k_N}\Big)\]
with equality whenever $P$ is such that, for all $N$ sufficiently large, the measure $\tilde P$ with density
\[f(x; \tilde P)= \frac{f(x; P)}{\lambda - \frac{k_N}{N}\phi(x)}\]
(where $\lambda>\sup_x\phi(x)$ is chosen to ensure this is a probability density) is also in $\mathcal{Q}$. (This can be thought of as related to the central limit theorem, cf. Example \ref{exampleNormal}.)
\end{enumerate}
\end{theorem}
\begin{remark}
Given the error of the expectation based on the $\mathcal{Q}$-MLE is asymptotically of the order of $1/\sqrt{N}$, the requirement implied by (ii) that $k_N$ grows faster than $\sqrt{N}$, is unsurprising, as this is what is needed to ensure that the risk aversion term $\frac{k_N}{2N}\mathrm{Var}(\xi)$ asymptotically dominates the statistical error of the estimation of $E_P[\xi]$.
\end{remark}
\begin{proof}
We begin by proving (ii). As $\mathcal{Q}$ is a GCD class, we know that, 
\[\Big|\frac{1}{N}\alpha_{\mathcal{Q}|\mathbf{x}_N}(Q) - D_{\mathrm{KL}|X}(P||Q)\Big|\leq O_P(N^{-1/2}),\]
with error bounded independently of $Q$. Hence, uniformly in $Q$,
\[\Big|\frac{1}{k_N}\alpha_{\mathcal{Q}|\mathbf{x}_N}(Q) - \frac{N}{k_N}D_{\mathrm{KL}|X}(P||Q)\Big|\leq O_P(N^{1/2}/k_N).\]
Calculating $\CE^{k_N,1}_{\mathcal{Q}|\mathbf{x}_N}(\xi)$, we have
\[\begin{split}
\CE^{k_N,1}_{\mathcal{Q}|\mathbf{x}_N}(\xi) &= \sup_{Q\in\mathcal{Q}}\Big\{E_Q[\xi] - \frac{1}{k_N}\alpha_{\mathcal{Q}|\mathbf{x}_N}(Q)\Big\}\\
& = \sup_{Q\in\mathcal{Q}}\Big\{E_Q[\xi] - \frac{N}{k_N}D_{\mathrm{KL}|X}(P||Q)\Big\}+O_P(N^{1/2}/k_N).\end{split}\]
We shall now focus on solving the problem under the assumption that the penalty is given by $(N/k_N)D_{\mathrm{KL}}(P||Q)$.

For fixed $N$, we can try and solve this simplified problem directly. Assuming the optimum will be attained with a measure denoted $Q^g$, this corresponds to finding the density $g=f(\cdot, Q^g)$. Calculus of variations yields
\[\phi + \frac{N}{k_N}\Big(\frac{g}{f(\cdot, P)} +\lambda\Big) = 0,\]
or equivalently 
\[g = \frac{f(\cdot, P)}{\lambda - \frac{k_N}{N}\phi},\]
where $\lambda$ is chosen to ensure $g$ is a density, that is, $E_P[(\lambda-\frac{k_N}{N}\xi)^{-1}]=1$. This requires $\lambda>\frac{k_N}{N}\sup_x\phi(x)$ (this is the reason we have assumed $\xi=\phi(X)$ is bounded). As the map $\lambda\mapsto (\lambda-\frac{k_N}{N}\phi(x))^{-1}$ is monotone, we also know that the corresponding value of $\lambda$ is unique and 
\[\lambda\in \Big[ 1+\frac{k_N}{N}\inf_x\phi(x),\quad  1+\frac{k_N}{N}\sup_x\phi(x)\Big].\]
This avoids inconsistency with the requirement $\lambda>\frac{k_N}{N}\sup_x\phi(x)$ whenever $N$ is large enough that $\frac{N}{k_N}>2\sup_x |\phi(x)|$.  For every fixed large $N$, we have a compact set of values for $\lambda$. Therefore, we can assume $(\lambda-\frac{k_N}{N}\xi)^{-1}$ is uniformly approximated by its Taylor series in $\lambda$ around $\lambda=1+\frac{k_N}{N}E_P[\xi]$. Furthermore, we immediately see the first approximation 
\[\lambda = 1+\frac{k_N}{N}E_P[\xi]+O(k_N/N).\]

Expanding the Taylor series of $(\lambda-\frac{k_N}{N}\xi)^{-1}$, we have 
\[1= E_P\Big[1 - \Big(\lambda-1 - \frac{k_N}{N} \xi\Big) +\Big(\lambda -1 - \frac{k_N}{N} \xi\Big)^2+...\Big]\]
or equivalently
\begin{equation}\label{eq:taylorstep}
\lambda=1 +\frac{k_N}{N}  E_P[\xi] +E_P\Big[\Big(\lambda -1 - \frac{k_N}{N} \xi\Big)^2\Big]+O\Big(E_P\Big[\Big(\lambda -1 - \frac{k_N}{N} \xi\Big)^3\Big]\Big).\end{equation}
Substituting our first approximation of $\lambda$ on the right hand side of \eqref{eq:taylorstep}, we have
\[\lambda = 1 + \frac{k_N}{N} E_P[\xi] +\Big(\frac{k_N}{N}\Big)^2\mathrm{Var}_P[\xi]+O\Big(\Big(\frac{k_N}{N}\Big)^2\Big).\]
Substituting this second approximation back into \eqref{eq:taylorstep}, we observe that the error can be taken to be $O((k_N/N)^3)$, rather than $O((k_N/N)^2)$. 

We can now approximate our convex expectation. We know that
\[\begin{split}
E_{Q^g}[\xi] &= E_P\Big[\frac{\xi}{\lambda-\frac{k_N}{N}\xi}\Big]=E_P\Big[\xi\Big(1-\frac{k_N}{N}(E_P[\xi]-\xi) +O((k_N/N)^2)\Big)\Big]\\
&=E_P[\xi] + \frac{k_N}{N}\mathrm{Var}_P[\xi] +O((k_N/N)^2)
\end{split}\]
and similarly
\[
E_P\Big[\log\Big(\frac{dP}{dQ^g}\Big)\Big]=E_{P}\Big[\log\Big(\lambda-\frac{k_N}{N}\xi\Big)\Big]=\frac{1}{2}\Big(\frac{k_N}{N}\Big)^2\mathrm{Var}_P[\xi] + O\Big(\Big(\frac{k_N}{N}\Big)^3\Big).
\]
Hence we can calculate the desired approximation
\begin{equation}\label{eq:approximationE1}
\begin{split}\CE^{k_N,1}_{\mathcal{Q}|\mathbf{x}_N}(\xi) 
&= \sup_{Q\in\mathcal{Q}}\Big\{E_Q[\xi] - \frac{N}{k_N}D_{\mathrm{KL}|X}(P||Q)\Big\}+O_P\Big(\frac{N^{1/2}}{k_N}\Big)\\
&\leq E_{Q^g}[\xi] - \frac{N}{k_N}E_P\Big[\log\Big(\frac{dP}{dQ^g}\Big)\Big]+O_P\Big(\frac{N^{1/2}}{k_N}\Big)\\
&=E_P[\xi]+\frac{k_N}{2N}\mathrm{Var}_P(\xi) + O\Big(\Big(\frac{k_N}{N}\Big)^2\Big)+O_P\Big(\frac{N^{1/2}}{k_N}\Big).\end{split}\end{equation}
with equality whenever $Q^g\in\mathcal{Q}$, as stated in (ii).

We now seek to reduce to the assumptions of (i). As increasing $k_N$ will only increase the (nonnegative) differences
\[\CE^{k_N,1}_{\mathcal{Q}|\mathbf{x}_N}(\xi)-E_{\hat Q_N}[\xi], \qquad E_{\hat Q}[\xi]+\CE^{k_N,1}_{\mathcal{Q}|\mathbf{x}_N}(-\xi)\]
and we know that $E_{\hat Q_N}[\xi]$ is consistent, we can assume that $N^{1/2}/k_N\to 0$ without loss of generality. Under this assumption, the right hand side of \eqref{eq:approximationE1} converges to $E_P[\xi]$, and hence we verify that $\CE^{k_N,1}_{\mathcal{Q}|\mathbf{x}_N}(\xi)\to_P E_P[\xi]$ as desired.
\end{proof}

\begin{remark}
Assuming that $\mathcal{Q}$ is sufficiently rich and $k_N/\sqrt{N}\to\infty$ (so the approximation of (ii) is useful) this result implies that, if we have simple estimators of the mean and variance of $\xi=\phi(X)$, for example the classical sample mean and variance of $\{\phi(X_n)\}_{n\in\bN}$ (which have error $O_P(N^{-1/2})$), then we have the asymptotic approximation
\[\CE_{\mathcal{Q}|\mathbf{x}}^{k_N,1}(\xi) \approx \widehat{E_P[\xi]}+\frac{k_N }{N}\frac{1}{2}\widehat{\mathrm{Var}_P(\xi)}.\]
It is well known that a mean-variance criterion is not a convex expectation in general, however retains convexity for Gaussian distributions (see, for example, \cite{Follmer2002}). In this setting, we can see that the central limit theorem renders our uncertainty approximately Gaussian, so no contradiction arises. 
\end{remark}

We will now consider the case $\gamma=\infty$. It is easy to check that the interval 
\[\mathcal{I}_N(\xi)=\Big[-\CE^{k,\infty}_{\mathcal{Q}|\mathbf{x}_N}(-\xi),\quad \CE^{k,\infty}_{\mathcal{Q}|\mathbf{x}_N}(\xi)\Big]\]
is a likelihood interval for $E[\xi]$, that is, it corresponds to the range of expectations under the measures in $\mathcal{Q}$ with likelihood at least $e^{-k}$. Such intervals are commonly used as generalizations of confidence intervals (see for example Hudson \cite{Hudson1971}, drawing on the well known results of Neyman and Pearson \cite{Neyman1933}). In this context, we shall see that a stronger property holds, as the confidence region is uniform in $\phi$. (See also Theorem \ref{thm:Wilksresult}.)
\begin{theorem}\label{thm:llninfcase}
Suppose $\mathcal{Q}$ is a GCD family and  $X,\{X_n\}_{n\in\bN}$ are iid under each $Q\in\mathcal{Q}$. Then if $k_N=o(N)$, the nonlinear expectation with $\gamma=\infty$ is a \emph{uniformly} consistent estimator, that is,
\[\sup_{\phi:|\phi|\leq 1} \Big\{\CE^{k_N,\infty}_{\mathcal{Q}|\mathbf{x}_N}(\phi(X))- E_P[\phi(X)]\Big\} \to_P 0 \quad \text{for all }P\in\mathcal{Q}.\]
\end{theorem}
\begin{proof}
Observe that 
\[\CE^{k_N,\infty}_{\mathcal{Q}|\mathbf{x}_N}(\phi(X)) = \sup_{Q:\alpha_{\mathcal{Q}|\mathbf{x}_N}(Q)\leq k_N}\{E_Q[\phi(X)]\}.\]
As $\mathcal{Q}$ is a GCD class, we know that for any $P\in \mathcal{Q}$,
\[\frac{1}{N}\alpha_{\mathcal{Q}|\mathbf{x}_N}(Q) = D_{\mathrm{KL}|X}(P||Q) + O_P(N^{-1/2})\]
and so, provided $k_N=o(N)$, 
\[\alpha_{\mathcal{Q}|\mathbf{x}_N}(Q) \leq k_N \qquad \Leftrightarrow \qquad D_{\mathrm{KL}|X}(P||Q) \leq \frac{k_N}{N}+ O_P(N^{-1/2}) = o_P(1)\]
with the terminal error uniform in $Q$. From Pinsker's inequality, looking only at the marginal law of $X$, we know that the total variation norm satisfies
\[\int |f(x;P)-f(x;Q)|dx = \big\|P|_{\sigma(X)}-Q|_{\sigma(X)}\big\|_{\mathrm{TV}} \leq \sqrt{2 D_{\mathrm{KL}|X}(P||Q)}.\]
Therefore,
\[\begin{split}
  &\sup_{\phi:|\phi|\leq 1} \Big\{\CE^{k_N,\infty}_{\mathcal{Q}|\mathbf{x}_N}(\phi(X))- E_P[\phi(X)]\Big\} \\
 & =\sup_{\phi:|\phi|\leq 1}\sup_{\{Q:\alpha_{\mathcal{Q}|\mathbf{x}_N}(Q) \leq k_N\}} \Big\{E_Q[\phi(X)]- E_P[\phi(X)]\Big\}\\
&=\sup_{\phi:|\phi|\leq 1}\sup_{\{Q:D_{\mathrm{KL}|X}(P||Q) \leq o_P(1)\}} \Big\{E_Q[\phi(X)]- E_P[\phi(X)]\Big\}\\
&\leq \sup_{\{Q:D_{\mathrm{KL}|X}(P||Q) \leq o_P(1)\}} \Big\{\big\|P|_{\sigma(X)}-Q|_{\sigma(X)}\big\|_{\mathrm{TV}}\Big\}\\
&\leq o_P(1).
  \end{split}
\]
It follows that the nonlinear expectation is a uniformly consistent estimator.
 \end{proof}

By a simple comparison, we also obtain consistency for all other $\gamma\in[1,\infty]$.

\begin{corollary}
If $\mathcal{Q}$ is a GCD class, $k_N = o(N)$ and $\gamma\in[1,\infty]$, the nonlinear expectation $\CE^{k,\gamma}_{\mathcal{Q}|\mathbf{x}_N}(\phi(\xi))$ is a consistent estimator of $E_P[\phi(\xi)]$. 
\end{corollary}
\begin{proof}
We know that the two extreme cases $\gamma=1$ and $\gamma=\infty$ are both consistent, as is the MLE $E_{\hat Q_N}[\phi(\xi)]$ (this follows, for example, from the fact $E_{\hat Q_N}[\xi] \in \mathcal{I}_N$, where $\mathcal{I}_N$ is as in Theorem \ref{thm:llninfcase}. Furthermore, for any $\gamma$, as $|x|^\gamma \geq \min\{|x|, |x|^\infty\}$, it is easy to check from the definition that 
\[E_{\hat Q_N}[\xi]\leq \CE^{k_N,\gamma}_{\mathcal{Q}|\mathbf{x}_N}(\xi)\leq \max\Big\{\CE^{k_N,1}_{\mathcal{Q}|\mathbf{x}_N}(\xi), \CE^{k_N,\infty}_{\mathcal{Q}|\mathbf{x}_N}(\xi)\Big\}.\]
 The result follows.
\end{proof}

\begin{remark}
It seems likely that $\CE^{k,\gamma}_{\mathcal{Q}|\mathbf{x}}(\xi)$ also has an interpretation in terms of perturbing $E_P[\xi]$ by a term of the order of $E_P\Big[|\xi-E[\xi]|/\sqrt{N}\Big]^{2\gamma/(2\gamma-1)}$. These types of perturbation are rarely considered in other settings, so such results appear to be of purely technical interest.
\end{remark}

\subsection{Parametric results}\label{sec:parametric}

We now suppose that $\mathcal{Q}$ is a class of measures coming from a `nice' parametric family. In this setting, we can obtain more precise asymptotics by considering the divergence as a function of the parameter, rather than as a function of the abstract space of probability measures. For simplicity, we shall consider an exponential family of measures, which is general enough for many applications, but gives sufficient structure to obtain tight results. We shall also assume throughout that, for every $Q\in\mathcal{Q}$, $X, \{X_n\}_{n\in\bN}$ are iid with density $f(\cdot;Q)$.
\begin{definition}
 A distribution is said to come from the exponential family (in natural parameters) if the density can be written
\[f(x;Q) = h(x)\exp\Big\{ \langle \theta, T(x)\rangle -A(\theta)\Big\}.\]
Here $\theta$ is the parameter of $Q$, and is in an open subset $\Theta$ of $\bR^d$ for some $d$, $T$ is a (vector of) sufficient statistics, $h$ is a normalization function and $A$ is the log-partition function. We assume that $\mathcal{Q}$ corresponds to all those measures with parameters in $\Theta$, and write $\theta_Q$ for the parameters of $Q$, $Q^\theta$ for the measure associated with $\theta$, and $E_{\theta}$ for $E_{Q^\theta}$, etc...
\end{definition}

 The key result we shall use is that $A$ is convex and smooth (in particular has a continuous third derivative). We shall in fact use the following, slightly stronger condition.
\begin{assumption}
\begin{enumerate}[(i)]
 \item The Hessian $\mathfrak{I}_{\theta} = \partial^2 A(\theta)$ (commonly known as the information matrix) is (strictly) positive definite at every point of $\Theta$. 
 \item The $\mathcal{Q}$-MLE exists and is consistent, with probability tending to $1$ as $N\to\infty$ (that is, for every $Q\in \mathcal{Q}$, a maximizer $\hat Q_N$ exists with $Q$-probability approaching $1$ and $\hat\theta_N=\theta_{\hat Q_N} \to_Q \theta_Q$).
\end{enumerate}
\end{assumption}
These assumptions can be justified using weak assumptions on the family considered, see for example Berk \cite[Theorem 3.1]{Berk1972}, Silvey \cite{Silvey1959} or the more general discussion of Lehmann \cite{Lehmann1999} (see also \cite{Lehmann1998}). For more advanced discussion of the theory of likelihood in exponential families, see Barndorff-Nielsen \cite{Barndorff-Nielsen1978}.

Observe that, whenever the $\mathcal{Q}$-MLE $\hat \theta$ exists, the divergence is given by 
\[\alpha_{\mathcal{Q}|\mathbf{x}_N}(\theta) = -\sum_{n=1}^N \langle \theta-\hat \theta_N, T(X_i)\rangle  + N\big(A(\theta)-A(\hat \theta_N)\big),\]
using the natural abuse of notation $\alpha_{\mathcal{Q}|\mathbf{x}_N}(\theta) := \alpha_{\mathcal{Q}|\mathbf{x}_N}(Q^\theta)$. Given a first order condition will hold at the MLE, we can simplify to remove dependence on the observations (except through the MLE)
\[\alpha_{\mathcal{Q}|\mathbf{x}_N}(\theta) = N\big(A(\theta)-A(\hat \theta_N)-\langle \theta-\hat\theta_N, \partial A(\hat \theta_N)\rangle\big).\]

The following result will allow us to get a tight asymptotic approximation of the penalty, as it will allow us to focus our attention on a small ball around the MLE.
\begin{lemma}\label{lem:unifBound}
Let $\rho>0$ be a constant and let $\hat \theta_N$ denote the MLE of $\theta$. Then, for each $P\in\mathcal{Q}$, there exist constants $c_1, c_2$ independent of $N$ such that, writing 
\[R=  \frac{c_1\rho}{N} \vee \sqrt{\frac{c_2 \rho}{N}} = O(N^{-1/2})\]
we have that 
\[P\big(\alpha_{\mathcal{Q}|\mathbf{x}}(\theta)>  \rho \text{ for all }\theta:\|\theta - \hat \theta\|> R\big)\to 1.\]
In other words, with high probability, we know $\alpha_{\mathcal{Q}|\mathbf{x}}>\rho$ whenever $\|\theta - \hat \theta\|> R=O(N^{-1/2})$.
\end{lemma}
\begin{proof}
 See Appendix.
\end{proof}

\begin{remark}\label{rem:unifBoundapply}
The previous result will mainly be used to show that, when we consider bounded random variables, for any $P\in\mathcal{Q}$ we can approximate the divergence by
\[\alpha_{\mathcal{Q}|\mathbf{x}_N}(\theta) = \frac{N}{2}(\theta-\hat \theta_N)^\top \Big[\mathfrak{I}_{\hat\theta_N}+O_P(N^{-1/2})\Big](\theta-\hat \theta_N).\]
This is itself an interesting and useful result, particularly when we use the DR-expectation approach as a first step in a larger problem. For example, when we use a DR-expectation to capture the uncertainty in calibration of a model, which we then wish to use in a variety of settings this result shows that it is enough (to first order) to penalize using the observed information matrix, rather than repeatedly calculating the likelihood function. This is the approximation we made in \eqref{eq:coinpenalty}. 

As the approximation is a quadratic, the optimization needed to calculate $\CE_{\mathcal{Q}|\mathbf{x}_N}^{k,\gamma}$ is straightforward (particularly for linear or quadratic functionals of the parameters),  which can have significant numerical advantages (see for example Ben-Tal and Nemirovski \cite{Ben-Tal1998}).
\end{remark}
We now use this approximation to give asymptotic estimates for the DR-expectation. This can be seen as an analogue to the central limit theorem (cf. Example \ref{exampleNormal}). Note that, unlike in the nonparametric case, we do not need to scale the risk aversion parameter $k$ as $N\to \infty$. It is convenient to make the following definition.

\begin{definition}
Let $\phi$ be a bounded function such that the map $\tilde\phi: \theta\mapsto E_\theta[\phi(X)]$ is differentiable. We write 
\[V(\phi, \hat\theta) := (\partial \tilde\phi|_{\hat \theta})^\top (\mathfrak{I}_{\hat\theta}^{-1})(\partial \tilde\phi|_{\hat \theta}).\]
\end{definition}
\begin{remark}
Observe that, by classical arguments, if $\phi$ can be written as a linear function of the sufficient statistics then 
\[V(\phi,\hat\theta) = \mathrm{Var}_{\hat \theta}(\phi(X)).\]
If $\hat\theta_N$ has the variance appearing in the central limit theorem, that is\footnote{See Lehmann \cite[Section 7.7]{Lehmann1999} for one set of sufficient conditions under which this holds.},   $\mathrm{Var}(\hat\theta_N) \approx N^{-1}\mathfrak{I}_{\theta_P}^{-1}$,  then (given an appropriate array of integrability and continuity assumptions) we have the approximate variance of the MLE-expectation
\[\frac{1}{N}V(\phi, \hat\theta)\approx \mathrm{Var}_P(E_{\hat\theta_N}[\phi(X)]).\]
\end{remark}

\begin{theorem}
 Let $\phi$ be a bounded function such that the map $\tilde\phi:\theta \mapsto E_{Q^\theta}[\phi(X)]$ is twice differentiable. Then for all $P\in\mathcal{Q}$,
\[\CE^{k,1}_{\mathcal{Q}|\mathbf{x}_N}(\phi(X)) = E_{\hat \theta_N}[\phi(X)] + \frac{k}{2N} V(\phi, \hat\theta_N) + O_P(N^{-3/2}).\]
\end{theorem}
\begin{proof}
Fix $P\in \mathcal{Q}$. For simplicity, we write $\hat \theta$ for $\hat \theta_N$. To begin, observe that 
\[\CE^{k,1}_{\mathcal{Q}|\mathbf{x}_N}(\phi(X)) = \sup_{\theta\in\Theta}\Big\{E_{Q^\theta}[\phi(X)] - \frac{1}{k}\alpha_{\mathcal{Q}|\mathbf{x}_N}(\theta)\Big\}\]
and as $\phi$ is bounded, we need only consider those measures
\[\Theta_N=\Big\{\theta\in\Theta:\alpha_{\mathcal{Q}|\mathbf{x}_N}(\theta)\leq k\sup_x|\phi(x)|\Big\}.\]
From Lemma \ref{lem:unifBound}, we know that 
\[P\Big(\sup_{\theta\in\Theta_N} \|\theta-\hat\theta\|>O(N^{-1/2})\Big) \to 0.\]

We know $\hat \theta\to_P \theta_P$ and $\tilde\phi$ is twice differentiable at $\theta_P$, so for $\theta\in\Theta_N$,
\[\begin{split}
   E_{Q^\theta}[\phi(X)] &= \tilde\phi(\hat\theta) + \Big\langle \theta - \hat\theta, \partial \tilde\phi|_{\hat\theta}+ O_P(\|\theta-\hat\theta\|)\Big\rangle \\
&= \tilde\phi(\hat\theta) + \Big\langle \theta - \hat\theta, \partial \tilde\phi|_{\hat\theta}+ O_P(N^{-1/2})\Big\rangle 
  \end{split}
\]
We also know that $\alpha_{\mathcal{Q}|\mathbf{x}_N}(\theta)$ is smooth, convex and minimized at $\hat\theta$, so for $\theta\in\Theta_N$,
\[\begin{split}
   \alpha_{\mathcal{Q}|\mathbf{x}_N}(\theta)&= \frac{N}{2}(\theta-\hat\theta)^\top \Big[\mathfrak{I}_{\hat\theta}+O_P(\|\theta-\hat\theta\|)\Big](\theta-\hat\theta)\\
&= \frac{N}{2}(\theta-\hat\theta)^\top \Big[\mathfrak{I}_{\hat\theta}+O_P(N^{-1/2})\Big](\theta-\hat\theta).
  \end{split}
\]
Substituting these, we have the approximate DR-expectation
\[\begin{split}
 & \CE^{k,1}_{\mathcal{Q}|\mathbf{x}_N}(\phi(X))\\
& = \tilde\phi(\hat\theta) + \sup_{\theta\in\Theta_N}\Big\{\Big\langle \theta - \hat\theta, \partial \tilde\phi|_{\hat\theta}+ O_P(N^{-1/2})\Big\rangle \\
&\qquad\qquad\qquad - \frac{N}{2k}(\theta-\hat\theta)^\top \Big[\mathfrak{I}_{\hat\theta}+O_P(N^{-1/2})\Big](\theta-\hat\theta)\Big\}.
  \end{split}
\]
The term in braces has optimizer 
\[\theta^* = \hat\theta + \frac{k}{N} \Big(\mathfrak{I}_{\hat\theta}+O_P(N^{-1/2})\Big)^{-1}\Big(\partial \tilde\theta|_{\hat \theta}+O_P(N^{-1/2})\Big),\]
where we know that, as $\hat\theta\to \theta_P$ and $\mathfrak{I}_{\theta^P}$ is positive definite, with $P$-probability approaching $1$ the matrix $\mathfrak{I}_{\hat\theta}+O_P(N^{-1/2})$ is nonsingular. Substituting, we have the desired approximation
\[\CE^{k,1}_{\mathcal{Q}|\mathbf{x}_N}(\phi(X)) = \tilde\phi(\hat\theta) + \frac{k}{2N} (\partial \tilde\phi|_{\hat \theta})^\top (\mathfrak{I}_{\hat\theta}^{-1})(\partial \tilde\phi|_{\hat \theta}) + O_P(N^{-3/2}).\]
\end{proof}

We now consider the case $\gamma=\infty$.

\begin{theorem}
 Let $\phi$ be a bounded function such that the map $\tilde\phi:\theta \mapsto E_{Q^\theta}[\phi(X)]$ is twice differentiable. Then for all $P\in\mathcal{Q}$,
\[\CE^{k,\infty}_{\mathcal{Q}|\mathbf{x}_N}(\phi(X)) = E_{\hat \theta_N}[\phi(X)] + \sqrt{\frac{2k}{N} V(\phi, \hat \theta_N)} + O_P(N^{-3/4}).\]
\end{theorem}
\begin{proof}
 The proof follows much in the same way as the case $\gamma=1$ and we use the same notation. We know that 
\[\CE^{k,\infty}_{\mathcal{Q}|\mathbf{x}_N}(\phi(X)) = E_{\hat\theta}[\phi(X)] + \sup_{\theta: \alpha_{\mathcal{Q}|\mathbf{x}_N}(\theta)\leq k}\Big\langle \theta - \hat\theta, \partial \tilde\phi|_{\hat\theta}+ O_P(\|\theta-\hat\theta\|)\Big\rangle.\]
We see that 
\[\alpha_{\mathcal{Q}|\mathbf{x}_N}(\theta)= \frac{N}{2}(\theta-\hat\theta)^\top \Big[\mathfrak{I}_{\hat\theta}+O_P(\|\theta-\hat\theta\|)\Big](\theta-\hat\theta)\]
and from Lemma \ref{lem:unifBound}, with probability approaching $1$, it is enough to consider to $\Theta_N=\{\theta: \|\theta-\hat\theta\|<O_P(N^{-1/2})\}$. Standard optimization then yields
\[\begin{split}
&\CE^{k,\infty}_{\mathcal{Q}|\mathbf{x}}(\phi(X)) -E_{\hat\theta}[\phi(X)]\\
    &= \sqrt{\frac{k}{2N}}\Big((\partial \tilde\phi|_{\hat\theta}+ O_P(N^{-1/2}))^{\top}\Big[\mathfrak{I}_{\hat\theta}+O_P(N^{-1/2})\Big]^{-1}(\partial \tilde\phi|_{\hat\theta}+ O_P(N^{-1/2}))\Big)^{1/2}\\
    &= \sqrt{\frac{k}{2N}}\Big((\partial \tilde\phi|_{\hat\theta})^{\top}\mathfrak{I}_{\hat\theta}^{-1}(\partial \tilde\phi|_{\hat\theta})\Big)^{1/2} + O_P(N^{-3/4}).
  \end{split}\]
The result follows.
\end{proof}

\begin{remark}
The cases $\gamma \in (1,\infty)$ can also be treated using the approximation implied by Lemma \ref{lem:unifBound} (in the way suggested by Remark \ref{rem:unifBoundapply}), and are left as a tedious exercise for the reader.
\end{remark}

The following result is can be shown to hold under the assumption that $\mathcal{Q}$ is of the exponential family we consider here, or more generally. It is particularly of interest as it is naturally a `uniform' result over the space of outcomes (which do not need to be bounded or independent of the observations). This is of importance in decision marking, as we will often wish to choose between a range of outcomes $\xi$, and wish to be confident that our comparison method is valid for all choices simultaneously.

\begin{theorem}\label{thm:Wilksresult}
Suppose the MLE is consistent and Wilks' theorem holds under every $P\in\mathcal{Q}$ (that is, $\frac{1}{2}\alpha_{\mathcal{Q}|\mathbf{x}_N}(P)$ is asymptotically $\chi^2_d$ distributed under $P$, where $d$ is a known parameter). Then, for a random variable $\xi$,
\[\mathcal{I}_N(\xi)=\Big[-\CE^{k,\infty}_{\mathcal{Q}|\mathbf{x}_N}(-\xi),\quad \CE^{k,\infty}_{\mathcal{Q}|\mathbf{x}_N}(\xi)\Big]\]
is a likelihood interval for $E[\xi]$, with the uniform asymptotic property 
\[\lim_N P\Big(E_P[\xi|\mathbf{x}_N]\in \mathcal{I}_N \text{ for all }\xi\Big)\geq F_{\chi^2_{d}}(2k).\]
\end{theorem}
\begin{proof}
That $\mathcal{I}_N(\xi)$ corresponds to a likelihood interval is trivial, as $\gamma=\infty$ implies we are considering expectations under those measures where the log likelihood (relative to the MLE) is at least $k$. Wilks' theorem then determines the asymptotic behaviour of the relative log likelihood, in particular, we know
\[P\big(\alpha_{\mathcal{Q}|\mathbf{x}_N}(P) \leq k\big) \to F_{\chi^2_d}(2k)\qquad \text{for all }P\in\mathcal{Q},\]
where $F_{\chi^2_d}$ is the cdf of the $\chi^2_d$-distribution. Clearly $\alpha_{\mathcal{Q}|\mathbf{x}_N}(P) \leq k$ implies $E_P[\xi|\mathbf{x}_N] \in \mathcal{I}_N(\xi)$ for all $\xi$. We then obtain the desired result,
\[P\Big(E_P[\xi|\mathbf{x}_N] \in \mathcal{I}_N\text{ for all }\xi\Big) \geq P\big(\alpha_{\mathcal{Q}|\mathbf{x}_N}(P) \leq k\big) \to F_{\chi^2_d}(2k).\]
\end{proof}

\begin{remark}
 Conditions for Wilks' theorem are closely related to those for the central limit theorem, and are typically based on integrability and continuity assumptions on the densities. The result is that,  
\[\frac{1}{2}\alpha_{\mathcal{Q}|\mathbf{x}}(P) \to_{P\text{-}\mathrm{Dist}} \chi^2_d\]
where $d$ is the dimension of the parameter space and ${P\text{-}\mathrm{Dist}}$ refers to convergence, in distribution, under $P$. See Lehmann \cite[Section 7.7]{Lehmann1999} for details.
\end{remark}

\section{Robustness and models}\label{sec:robust}
In this section, we shall consider the behaviour of the divergence-robust expectation for \emph{unbounded} random variables, and its relationship with `robust' statistical estimates. We shall regard the sample size $N$ as fixed. The following theorem complements our earlier asymptotic results (which were generally for bounded outcomes), to demonstrate that without any parametric structure most unbounded random variables do not have finite DR-expectations.

\begin{theorem}
Let $\mathcal{Q}$ be a family of measures such that $\mathcal{Q}$ is closed under taking finite mixtures (i.e. finite convex combinations of measures). Then for any random variable $\xi$ such that $\sup_{Q\in\mathcal{Q}}\{E_Q[\xi]\}=\infty$, for any $\gamma\in[1,\infty]$,
\[\CE_{\mathcal{Q}|\mathbf{x}}^{k, \gamma}(\xi) = \infty.\]
\end{theorem}

\begin{proof}
 For any $\epsilon>0$, let $Q^\epsilon\in \mathcal{Q}$ be a measure such that $E_{Q^\epsilon}[\xi]>\epsilon^{-2}$. For any measure $P\in \mathcal{Q}$, we define the mixture distribution $P(\epsilon) = (1-\epsilon)P+\epsilon Q^{\epsilon}$. It follows that $P(\epsilon)\in \mathcal{Q}$ and, provided $E_{P}[\xi]>-\infty$,
\[E_{P(\epsilon)}[\xi] = (1-\epsilon)E_{P}[\xi]+\epsilon E_{Q^\epsilon}[\xi] \geq (1-\epsilon)E_{P}[\xi]+\epsilon^{-1} \to \infty\]
as $\epsilon\to 0$. Also, we know $L(P(\epsilon)|\mathbf{x}) = (1-\epsilon)L(P|\mathbf{x}) + \epsilon L(Q^\epsilon|\mathbf{x}) > (1-\epsilon)L(P|\mathbf{x})$, so (assuming for notational simplicity that the $\mathcal{Q}$-MLE $\hat Q$ exists)
\[\alpha_{\mathcal{Q}|\mathbf{x}}(P(\epsilon)) = -\log\Big(\frac{L(P(\epsilon)|\mathbf{x})}{L(\hat Q|\mathbf{x})}\Big) < -\log\Big(\frac{(1-\epsilon)L(P|\mathbf{x})}{L(\hat Q|\mathbf{x})}\Big)\to \alpha_{\mathcal{Q}|\mathbf{x}}(P)<\infty.\]
It follows that, as $\epsilon \to 0$,
\[\CE_{\mathcal{Q}|\mathbf{x}}^{k,\gamma}(\xi) \geq (1-\epsilon)E_{P}[\xi]+\epsilon^{-1}-\Big(\frac{1}{k}\log\Big(\frac{(1-\epsilon)L(P|\mathbf{x})}{L(\hat Q|\mathbf{x})}\Big)\Big)^\gamma \to \infty.\]
\end{proof}

\begin{remark}
 The above assumes $\mathcal{Q}$ is closed under finite mixtures of measures. If we assume that $\mathcal{Q}$ is such that $\{X_n\}_{n\in\bN}$ are iid, then this is not the case. However, for $N<\infty$, an almost identical proof holds whenever $\mathcal{Q}$ is associated with a family of densities $f(\cdot;Q)$, and this family of densities is closed under taking finite mixtures. (The only significant change is that we obtain the inequality $L(P(\epsilon)|\mathbf{x}_N) > (1-\epsilon)^NL(P|\mathbf{x})$, where $P(\epsilon)$ corresponds to the measure with a mixture density.)
\end{remark}

This result highlights the importance of parametric structure for estimation of unbounded random variables, in terms of restricting the class of probability measures that can be considered. This restriction can be thought of in terms of restricting the probabilities of very large (positive or negative) values of $\xi$, and hence ensuring enough integrability that finite expectations arise. Without these restrictions, unlikely events (which by their very nature will generally not be seen in the data, so are not penalized) result in unbounded expectations. 

\begin{remark}
While we have not considered the numerical aspects of this problem in great detail, it is often the case that parametric models are also needed to reduce our problem to a finite-dimensional setting, rather than needing to solve optimization problems on the infinite-dimensional space of measures. 
\end{remark}

Given the importance of parametric families, it is then of interest to consider how the `statistical robustness' of the parametric estimation problem interacts with the `robustness' of the expectations considered. Given our use of likelihood theory, there is a natural connection to $M$-estimates, which correspond to estimates obtained by maximizing some function. Before giving general results, we consider a simple setting.

\begin{example}
Consider $X,\{X_n\}_{n=1}^N$ iid observations from a Laplace (or double exponential) distribution with known scale $1$ and unknown mean $\mu$. That is, $X_n$ has density
\[f(x) = \frac{1}{2} \exp(-|x-\mu|).\]
Let $\mathcal{Q}$ denote the corresponding family of measures and write $Q^\mu$ for the measure with mean $\mu$. For simplicity, assume $N$ is odd, so the MLE is uniquely given by $Q^\mathbf{m}$, where $\mathbf{m}$ is the sample median. This is known to be `statistically robust', see Huber and Ronchetti \cite{Huber2009}, as it does not depend on extreme observations, and is therefore unaffected by outliers.

The $\mathcal{Q}|\mathbf{x}_N$-divergence is then given by
\[\alpha_{\mathcal{Q}|\mathbf{x}_N}(Q^\mu) = \sum_{n=1}^N \Big(|X_n-\mu|-|X_n-\mathbf{m}|\Big).\]
For $X$ an iid observation from the same distribution as $X_n$ and $\beta>0$ (the case $\beta<0$ is symmetric) we have
\[\CE_{\mathcal{Q}|\mathbf{x}_N}^{k,1}(\beta X) = \sup_\mu\Big\{\beta\mu - \frac{1}{k}\sum_{n=1}^N \Big(|X_n-\mu|-|X_n-\mathbf{m}|\Big)\Big\}.\]
A first observation which can be drawn is that $\CE_{\mathcal{Q}|\mathbf{x}}^{k,1}(\beta X)$ is generally infinite, unless $\beta< N/k$. To see this, observe that if $\beta>N/k$, then the function to maximize is linear and increasing for $\mu>\max_{n\leq N} X_n$.

Assuming that $\beta<N/k$, the function to maximize is piecewise linear, concave and asymptotically decreasing (for both positive and negative $\mu$), so a finite solution exists. Except at points where $\mu=X_n$ for some $n$, we can differentiate to obtain the equation
\[0\bowtie \beta - \frac{1}{k}\sum_{n=1}^N (I_{\{X_n<\mu\}}- I_{\{X_n>\mu\}})=:\beta-\frac{N}{k}G(\mu)\]
where $\bowtie$ indicates that either the statement is an equality, or that the left and right limits of the right hand side differ in sign (if $X_n=\mu$ for some $n$). As we are looking for the maximal solution, we can generally state that the solution will be
\[\mu^* = \inf\Big\{\mu: \beta-\frac{N}{k} G(\mu) >0\Big\}.\]
We can also write
\[G(\mu) = (1-F(\mu)) - F(\mu-)\]
where $F(y) = \frac{1}{N} \sum_{n=1}^N I_{\{X_n\leq y\}}$ is the empirical cdf of our observations. Assuming $N$ is moderately large, this is well approximated by a continuous increasing function (so all quantiles are uniquely defined), and we will obtain
\[\mu^* \approx F^{-1}\Big(\frac{1}{2} + \frac{\beta k}{2N}\Big).\]
It follows that the optimizing choice of $\mu^*$ is given by the empirical $\frac{1}{2} + \frac{\beta k}{2N}$ quantile.

Introducing this back into our equation for $\CE_{\mathcal{Q}|\mathbf{x}_N}^{k,1}(\beta X)$, we obtain
\[\begin{split}
  \CE_{\mathcal{Q}|\mathbf{x}_N}^{k,1}(\beta X) &= \beta\mu^* - \frac{1}{k}\sum_{n=1}^N \Big(|X_n-\mu^*|-|X_n-\mathbf{m}|\Big)\\
&= \beta\mu^* - (\mu^*-\mathbf{m})\sum_{n=1}^N (I_{\{X_n\leq\mathbf{m}\}} - I_{\{X_n>\mu^*\}})\\&\qquad + \sum_{n=1}^N I_{\{\mathbf{m}<X_n\leq\mu^*\}}\Big(\frac{\mu^*+\mathbf{m}}{2}-X_n\Big)\\
&\approx \beta\mu^*- \Big(\mu^*-\mathbf{m}-\frac{\mu^*+\mathbf{m}}{2}\Big)\frac{k\beta}{2N} - \sum_{n=1}^N I_{\{\mathbf{m}<X_n\leq\mu^*\}}X_n\\
&\approx \Big(1-\frac{k}{4N}\Big) \beta\mu^*+\frac{3k}{4N}\beta\mathbf{m} - \frac{k}{2N}\frac{\sum_{n=1}^N I_{\{\mathbf{m}<X_n\leq\mu^*\}}(\beta X_n)}{\sum_{i} I_{\{\mathbf{m}<X_n\leq\mu^*\}}}.
 \end{split}
\]
We see that the divergence-robust estimate depends on a weighted combination of the median $\beta\mathbf{m}$, an upper quantile $\beta\mu^*$, and the mean value taken between these two bounds\footnote{This estimate can then be compared with the various perturbations of value-at-risk considered by Cont, Deguest and Scandolo \cite{Cont2010}. However, it is important to note that this closed-form is only for the random variables $\beta X$, not for a general random variable.}. Therefore, this quantity can still be robustly estimated, as it still does not depend on the tail behaviour beyond the $\frac{1}{2} + \frac{\beta k}{2N}$ quantile. More formally, the breakdown point of this estimator (the proportion of the data which can be made arbitrarily large without affecting the estimate) is $\frac{1}{2}(1-\beta \frac{k}{N})$. 

It is easy to see (as $E_{Q^\mu}[\beta X^2] =\beta\mu^2 + 2\beta$) that
\[\CE_{\mathcal{Q}|\mathbf{x}}^{k,1}(\beta X^2) =\infty\]
 for all $N,k, \beta>0$. For negative $\beta$, a finite answer can be obtained, but even its approximate closed-form representation is inelegant.
\end{example}

Comparing this example with the normal example (Example \ref{exampleNormal}), we can see that, when considering a likelihood model, there is a delicate relationship between the `statistical' robustness in the classical estimation problem and the `parameter uncertainty' robustness embedded in $\CE_{\mathcal{Q}|\mathbf{x}}^{k,1}$. The following theorem makes this behaviour more precise.

\begin{theorem}\label{thm:robustequiv}
Consider a sequence of iid random variables $X, \{X_n\}_{n=1}^N$, and a family of measures $\mathcal{Q}$ describing an uncertain `location parameter'. In other words, under $Q\in \mathcal{Q}$, suppose $X,\{X_n\}_{n=1}^N$ are iid observations from a distribution with density $\exp(\Psi(x-\mu_Q))$, and so $\mathcal{Q}$ is parameterized by $\mu_Q=E_Q[X]\in \mathbb{R}$. Suppose $\Psi$ has monotone increasing derivative $\psi$ (which may be discontinuous) and 
\[-\infty\leq \lim_{x\to -\infty}\psi(x) <0<\lim_{x\to +\infty} \psi(x)\leq \infty.\]
Note that the MLE parameter (assuming it exists, for simplicity) is given by the solution $\mu$ to $\sum_n \psi(X_n - \mu) = 0$.

The following are then equivalent.
\begin{enumerate}[(i)]
 \item The MLE parameter has a breakdown point above zero (that is, some fraction of the observations can be made arbitrarily large or small without making the MLE arbitrarily large or small),
 \item The MLE parameter is weakly continuous with respect to the empirical cdf of observations, for any empirical cdf where the MLE parameter is uniquely defined,
 \item $\psi$ is bounded,
 \item For any fixed $k,N$, for all $\beta\in \bR$ sufficiently large (in absolute value), $\CE_{\mathcal{Q}|\mathbf{x}_N}^{k,1}(\beta \xi)\not\in \bR$, where $\xi$ is an iid copy of the observations.
\end{enumerate}
\end{theorem}
\begin{proof}
 The equivalence of (i)-(iii) is given by Chapter 3 of Huber and Ronchetti \cite{Huber2009}, in particular Theorem 3.6. We seek to show (iii) and (iv) are equivalent. First, if (iii) holds, then $\Psi$ is of linear growth. Let $\beta>N\sup_x|\psi(x)|$.  We can then calculate
\[\CE_{\mathcal{Q}|\mathbf{x}}(\beta X)= \sup_{\mu_Q}\Big\{\beta\mu_Q - \sum_n\Big(\Psi(X_n-\mu_Q)-\Psi(X_n-\mu_{\mathrm{MLE}})\Big)\Big\}.\]
As $\beta$ is larger than the maximal derivative of $\sum_n \Psi(X_n-\mu_Q)$, we can see that the term in brackets is unbounded, so (iv) holds. A similar result holds if $\beta<-N\sup_x|\psi(x)|$.

To show (iv) implies (iii), we first observe that (iv) implies that for all $\beta$ sufficiently large,
\[\sup_{\mu_Q}\Big\{\beta\mu_Q - \sum_n \Big(\Psi(X_n-\mu_Q)-\Psi(X_n-\mu_{\mathrm{MLE}})\Big)\Big\} = \infty.\]
In other words, $\Psi$ is bounded above by a linear function. As $\psi=\Psi'$ is monotone increasing, this implies that $\psi$ is bounded above. A similar argument shows that $\psi$ is bounded below.
\end{proof}

\begin{remark}
This behaviour may seem pathological, but it has a natural interpretation. Suppose one has an estimation technique which does not depend on some property of the data, for example does not depend on some extreme quantile. Then, given the results of fitting this method, one should not be able to make strong statements about the behaviour of future observations at this quantile. 

In some sense, this is what is being captured by our approach. In the Laplace distributed case, our MLE does not depend on the behaviour of the data in the tails of the distribution (particularly not in a `linear' way, unlike the sample mean), so it is unsurprising that we at times cannot say much about the mean of linear functions of future observations. 

At the same time, this interpretation of the non-existence of moments is imperfect, as we still do have finite moments for sufficiently small multiples of $X$.
\end{remark}

Conversely, we have the following.
\begin{theorem}
Suppose we are in the setting of Theorem \ref{thm:robustequiv} and write 
\[\psi(\pm\infty) = \lim_{x\to\pm\infty}\psi(x).\]
If $\psi(-\infty)<\pm\beta/N<\psi(\infty)$, then the nonlinear expectation $\CE_{\mathcal{Q}|\mathbf{x}_N}^{k,1}(\beta X)$ is finite, and has breakdown fraction at least
\[\delta=\min\Big\{\frac{\psi(\infty) - |\beta|/N}{\psi(\infty)-\psi(-\infty)}, \frac{-|\beta|/N-\psi(-\infty)}{\psi(\infty)-\psi(-\infty)}\Big\}.\]
That is, for $m/N<\delta$, at least $m$ observations can be made arbitrarily large or small while $\CE_{\mathcal{Q}|\mathbf{x}_N}^{k,1}(\beta X)$ remains bounded.
\end{theorem}
\begin{proof}
Without loss of generality, suppose $\beta>0$ (we shall prove the result for both $\beta$ and $-\beta$ simultaneously). 
Consider the functions 
\[\lambda^{\pm}(\mu,\mathbf{x}):=\pm\frac{\beta}{N}-\frac{1}{N}\sum_{n=1}^N \psi(X_n-\mu).\]
From a first order condition, the value of $\pm\CE_{\mathcal{Q}|\mathbf{x}}^{k,1}(\pm\beta X)$ is given by 
\[\pm\CE_{\mathcal{Q}|\mathbf{x}}^{k,1}(\pm\beta X) = \pm\beta\mu^\pm(\mathbf{x}) - \sum_{n=1}^N \Big(\Psi(X_n - \mu^\pm(\mathbf{x}))-\Psi(X_n - \mu_\mathrm{MLE}(\mathbf{x}))\Big)\]
where $\mu^\pm(\mathbf{x})$ is the solution to $\lambda^\pm(\mu, \mathbf{x})\bowtie 0$ (where again $\bowtie$ indicates either equality or a change of sign) and $\mu_{\mathrm{MLE}}(\mathbf{x})$ is the MLE based on $\mathbf{x}$. 

Now observe that $\lambda^\pm$ is monotone increasing with respect to $\mu$ and, as $\psi(-\infty)<\beta/N<\psi(\infty)$, we know $\lim_{\mu\to\infty}\lambda^\pm(\mu, \mathbf{x})>0$ and $\lim_{\mu\to\infty}\lambda^\pm(\mu, \mathbf{x})<0$. Therefore, there is exactly one (finite) solution to $\lambda^\pm(\mu, \mathbf{x})\bowtie 0$. It follows that $\CE_{\mathcal{Q}|\mathbf{x}}^{k,1}(\pm\beta X)$ exists and is real.

We now need to determine the breakdown fraction. For $M$ a set of indices, let $\mathbf{x}(M,y)$ denote the set of observations, with $X_i$ replaced by $y_i$ for $i\in M$. Suppose $|M|=m$ and $m/N<\delta$. We wish to show that there is a bound on $\CE_{\mathcal{Q}|\mathbf{x}_N}^{k,1}(\pm\beta X)$ which is uniform in $y$. From the definition and nonnegativity of the penalty function, it is easy to see that 
\[\beta\mu^-(\mathbf{x}(M,y))\leq \pm\CE_{\mathcal{Q}|\mathbf{x}_N}^{k,1}(\pm\beta X)\leq \beta\mu^+(\mathbf{x}(M,y)),\]
it follows that it is enough for us to prove that $\mu^\pm(\mathbf{x}(M,y))$ is uniformly bounded in $y$.

As $\psi$ is monotone, we observe that 
\[\begin{split}
&\pm\frac{\beta}{N}-\frac{1}{N}\sum_{n\not\in M} \psi(X_n-\mu) -\frac{m}{N} \psi(\infty)\\
&\leq \lambda^\pm(\mu, \mathbf{x}(M,y))\\
&\leq \pm\frac{\beta}{N}-\frac{1}{N}\sum_{n\not\in M} \psi(X_n-\mu) -\frac{m}{N} \psi(-\infty).
\end{split}\]
By monotonicity, it is enough to show that the terms on the right and left have roots for finite values of $\mu$ (as these will not depend on $y$). Considering the lower bound first, we see that as $\mu\to\infty$, we obtain
\[\pm\frac{\beta}{N}-\frac{1}{N}\sum_{n\not\in M} \psi(X_n-\mu) -\frac{m}{N}\psi(\infty) \to \pm\frac{\beta}{N}-\Big(1-\frac{m}{N}\Big)\psi(-\infty) -\frac{m}{N} \psi(\infty)>0,\]
and as $\mu\to-\infty$,
\[\pm\frac{\beta}{N}-\frac{1}{N}\sum_{n\not\in M} \psi(X_n-\mu) -\frac{m}{N}\psi(\infty) \to \pm\frac{\beta}{N}-\psi(\infty)<0.\]
Therefore, there is a finite root for the lower bound on $\lambda^\pm(\mu, \mathbf{x}(M,y))$. A similar argument applies to the upper bound. By monotonicity, we conclude that $\mu^\pm(\mathbf{x}(M,y))$ and hence $\CE_{\mathcal{Q}|\mathbf{x}_N}^{k,1}(\pm\beta X)$ are uniformly bounded in $y$, as desired.
\end{proof}

\begin{remark}
Given the close relationship between entropy and extreme value theory, these results suggest a further relationship between the extremes of a class of models, the statistical robustness of estimators, and the existence of divergence-robust estimates. The development of this theory may be of future interest.
\end{remark}

To conclude, we observe that the non-existence of finite values for $\CE^{k,1}_{Q|\mathbf{x}}(X)$ can also manifest itself in surprising ways, as we can see from the following extension of Example \ref{exampleNormal}.

\begin{example}\label{meanvariancexample}
Consider the case where $X,\{X_i\}_{i=1}^N$ are iid $N(\mu, \sigma^2)$, where both $\mu$ and $\sigma^2$ are unknown. The divergence penalty is then (writing $\hat\sigma^2 = \frac{1}{N}\sum_{n=1}^N(X_n-\bar X)^2$)
\[\begin{split}
\alpha_{\mathcal{Q}|\mathbf{x}_N}(Q^{\mu,\sigma^2})&= \frac{N}{2}\log(\sigma^2) +\sum_{n=1}^N \frac{(X_n-\mu)^2}{2\sigma^2}  -\frac{N}{2}\log(\hat\sigma^2) -\sum_{n=1}^N \frac{(X_N-\bar x)^2}{2\hat\sigma^2} \\
&= \frac{N}{2}\Big(\log(\sigma^2/\hat\sigma^2)+\frac{\frac{1}{N}\sum_{n=1}^N (X_n-\mu)^2}{\sigma^2}-1\Big).
\end{split}\]
If we attempt to calculate $\CE_{\mathcal{Q}|\mathbf{x}}^{k,1}(\beta X)$, we obtain
\[\begin{split}
\CE_{\mathcal{Q}|\mathbf{x}_N}^{k,1}(\beta X)&=\sup_{\mu,\sigma^2}\Big\{\beta\mu - \frac{N}{2k}\Big(\log(\sigma^2/\hat\sigma^2)+\frac{\frac{1}{N}\sum_{n=1}^N (X_n-\mu)^2}{\sigma^2}-1\Big)\Big\}\\
&= \sup_{\sigma^2} \Big\{\beta \bar X + \frac{\beta^2 k}{2N}\sigma^2 - \frac{N}{2k}\Big(\log(\sigma^2/\hat\sigma^2)-1+\frac{\hat\sigma^2}{\sigma^2}\Big)\Big\}.
\end{split}\]
This causes a problem, as the term on the right is unbounded above with respect to $\sigma^2$. Looking more closely, this function typically has a local maximum for $\sigma^2\approx \hat\sigma^2$, but for very large values of $\sigma$ the $\frac{\beta^2 k}{2N}\sigma^2$ term will dominate. Therefore, there is no way that, even for large samples, a finite value of  $\CE_{\mathcal{Q}|\mathbf{x}_N}^{k,1}(\beta X)$ can be obtained.

One possible way to deal with this is to modify our approach slightly, either by including a prior distribution\footnote{In this case, to ensure a finite answer, the prior distribution would need to be asymptotically exponentially small as $\sigma^2\to\infty$, which is not the case for the conjugate inverse-Gamma distributions for $\sigma^2$. This renders explicit calculation difficult.}, or by adding an additional regularizing term to ensure the supremum chooses values close to the statistical parameters. For example, taking the penalty 
\[\tilde\alpha_{\mathcal{Q}|\mathbf{x}_N}(Q^{\mu,\sigma^2}) = \frac{N}{2}\Big(\log(\sigma^2/\hat\sigma^2)+\frac{\frac{1}{N}\sum_{n=1}^N (X_n-\mu)^2}{\sigma^2}-1\Big)+\epsilon(\sigma^2-\hat\sigma^2)\]
for some $\epsilon>0$, results in a finite value for $\CE_{\mathcal{Q}|\mathbf{x}_N}^{k,1}(\beta X)$ whenever $\frac{\beta^2k}{2N} < \frac{\epsilon}{k}$, in particular one obtains consistent estimates as $N\to\infty$. 
\end{example}

There are innumerable other applications and examples of this approach, and extensions to other settings (for example where likelihoods are replaced by more general objects) may also be of interest. While our results have focussed on the (analytically simpler) setting of independent observations, the approach naturally extends to where models include correlated observations, as is common in time-series models.  

\section*{Appendix}

\begin{proof}[Proof of Lemma \ref{lem:GCDconditions}]
 We know that 
 \[\alpha_{\mathcal{Q}|\mathbf{x}_N}(Q) = -\sum_{n=1}^N \log\Big(\frac{f(X_n; Q)}{f(X_n; P)}\Big) + \sum_{n=1}^N \log\Big(\frac{f(X_n; \hat Q_N)}{f(X_n; P)}\Big).\]
Considering the first term, by translation and scaling, we can assume that $K=[0,1]$. For any $P\in\mathcal{Q}$, write $F_P(x) = \int_{0}^x f(y;P)dy$ for the distribution function associated with $P$. We know that $F_P^{-1}$ is $C^1$ with a norm on its derivative independent of $P$, by assumption (ii). Next observe that the natural logarithm is $C^\infty$ on $[C^{-1},C]$, so by standard results on composition of functions, the map 
\[u\mapsto\ell(u;Q,P) := \log\Big(\frac{f(F_P^{-1}(u), Q)}{f(F_P^{-1}(u), P)}\Big)\]
is $\rho$-H\"older continuous, with a norm independent of $P,Q$.   Working under $P$, we note that $U_n = F_P(X_n)$ are independent and uniformly distributed on $[0,1]$, and $E_P[\ell(U; Q,P)] = D_{\mathrm{KL}|X}(P||Q)$. 

By rescaling, we can assume, without loss of generality, that for every $P,Q\in \mathcal{Q}$, 
\[\ell(\cdot;Q,P)\in \mathbf{F}_\rho = \{g: |g(u)-g(v)|\leq |u-v|^\rho \text{ for all } u,v\in[0,1]\}.\]
It is enough, therefore, to prove a uniform convergence rate for functions in $\mathbf{F}_\rho$.

We can now appeal to Corollary 17.3.3 and the proof of Theorem 17.3.1 of Shorack and Wellner \cite[p.633]{Shorack1986} (itself based on Strassen and Dudley \cite{Strassen1969}) to see that, writing
\[\sqrt{N}\Big(\frac{1}{N}\sum_n g(U_n) - E[g(U)]\Big)=:Z_N(g),\]
we know that for any $\eta>0$, there is $M$ sufficiently large (independent of $N$) that
\[P\Big(\sup_{g\in \mathbf{F}_\rho} \|Z_N(g)\| > M\Big)\leq \eta.\]
(The usual purpose of this is as a step towards showing that $Z_N$ converges weakly to a Gaussian process, which is a form of Donsker's theorem.) By rearranging, it follows that 
\begin{equation}\label{eq:llnbound1}
\gamma(N):=\sup_{Q\in\mathcal{Q}}\Big\{\Big|-\frac{1}{N}\sum_n \log\Big(\frac{f(X_n; Q)}{f(X_n; P)}\Big)- D_{\mathrm{KL}|X}(P||Q)\Big|\Big\}=O_P(N^{-1/2})
\end{equation}

In particular, we know that $\hat Q_N$ takes values in $\mathcal{Q}$, so,
\begin{equation}\label{eq:llnbound2}
\Big|-\frac{1}{N}\sum_{n=1}^N \log\Big(\frac{f(X_n; \hat Q_N)}{f(X_n; P)}\Big)- D_{\mathrm{KL}|X}(P||\hat Q_N)\Big|=O_P(N^{-1/2}).
\end{equation}
From the definition of $\hat Q_N$ we see that 
\[\begin{split}0 &= \frac{1}{N}\sum_{n=1}^N \log\Big(\frac{f(X_n; P)}{f(X_n; P)}\Big)\\
&\leq \frac{1}{N}\sum_{n=1}^N \log\Big(\frac{f(X_n; \hat Q_N)}{f(X_n; P)}\Big) = \sup_{Q\in\mathcal{Q}}\frac{1}{N}\sum_{n=1}^N \log\Big(\frac{f(X_n; Q)}{f(X_n; P)}\Big)\\
&\leq -D_{\mathrm{KL}|X}(P||\hat Q_N) + \gamma(N).\end{split}\]
Therefore,
\begin{equation}\label{eq:llnbound3}
0\leq D_{\mathrm{KL}|X}(P||\hat Q_N) \leq \gamma(N)=O_P(N^{-1/2})\end{equation}
The result then follows from using \eqref{eq:llnbound1}, \eqref{eq:llnbound2} and \eqref{eq:llnbound3} with the triangle inequality.
\end{proof}

\begin{proof}[Proof of Lemma \ref{lem:unifBound}]
Our proof depends on three facts: that $\alpha$ is locally a quadratic to second order (via Taylor's theorem), that the MLE is consistent (allowing us to bound the third derivative with high probability), and that $\alpha$ is convex (which controls its global behaviour). We write $\hat \theta$ for $\hat \theta_N$ for notational simplicity.

As the MLE is consistent (and exists with high probability), as $N\to \infty$, for any radius $C>0$, we know 
\begin{equation}\label{eq:localprobbound}
 P(\|\hat\theta-\theta^P\|<C/2))\to 1.
\end{equation}
We also know that, for some constant $k$ (which will in general depend on $P$ and on $C$ being sufficiently small, but is independent of $N$), we have the bound $\|\partial^3 A(\theta)\|\leq k$ for all $\theta$ with $\|\theta-\theta^P\|<C$. Combining these, for all $\theta$ with $\|\theta-\hat\theta\|<C/2$, from Taylor's theorem
\[\alpha_{\mathcal{Q}|\mathbf{x}}(\theta) \geq N \Big(\frac{1}{2} (\theta-\hat\theta)^\top \mathfrak{I}_{\hat \theta} (\theta-\hat\theta)-k\|\theta-\hat\theta\|^3\Big).\]

As we know that $\mathfrak{I}_{\hat \theta}$ is not degenerate (uniformly in a neighbourhood of $\theta^P$), we can also assume that (making $k$ sufficiently large) 
\[(\theta-\hat\theta)^\top \mathfrak{I}_{\hat \theta} (\theta-\hat\theta) \geq \frac{1}{k}\|\theta-\hat\theta\|^2.\]
Therefore, taking $C\leq k^{-2}$, on the set $\|\theta-\hat\theta\|\leq C/2$ we have
\begin{equation}\label{eq:alphabound1}
\begin{split}
   \alpha_{\mathcal{Q}|\mathbf{x}}(\theta) &\geq N\Big(\frac{1}{k}\|\theta-\hat\theta\|^2 - k\|\theta-\hat\theta\|^3\Big)\\
&\geq N\Big(\frac{1}{k} - \frac{kC}{2}\Big)\|\theta-\hat\theta\|^2 = \frac{N}{k}\Big(1 - \frac{k^2C}{2}\Big)\|\theta-\hat\theta\|^2\\
&\geq \frac{N}{2k}\|\theta-\hat\theta\|^2
  \end{split}
\end{equation}
Note that $k$ and $C$ do \emph{not} depend on $N$, so \eqref{eq:localprobbound} remains valid.

We now need to extend the bound of \eqref{eq:alphabound1} to all $\theta$. We know that $\alpha_{\mathcal{Q}|\mathbf{x}}$ is convex and $\alpha_{\mathcal{Q}|\mathbf{x}}(\hat\theta)=0$. For any point $\theta$ such that $\|\theta-\hat\theta\|> C/2$, its projection on the ball of radius $C/2$ around $\hat \theta$ is given by
\[\theta_\pi = \hat\theta + \lambda (\theta - \hat \theta) := \hat\theta + \frac{C}{2\|\theta - \hat \theta\|} (\theta - \hat \theta).\]
Hence, from \eqref{eq:alphabound1}, we know that 
\begin{equation}\label{eq:alphabound2}
 \alpha_{\mathcal{Q}|\mathbf{x}}(\theta) \geq \frac{1}{\lambda}\alpha_{\mathcal{Q}|\mathbf{x}}(\theta_\pi) \geq \frac{1}{\lambda}\frac{N}{2k}\|\theta_\pi-\hat\theta\|^2 = \frac{N}{Ck}\|\theta-\hat\theta\|.
\end{equation}
Combining \eqref{eq:alphabound1} and \eqref{eq:alphabound2}, we know that 
\[\alpha_{\mathcal{Q}|\mathbf{x}}(\theta) \geq \Big(\frac{N}{2k}\|\theta-\hat\theta\|^2\Big)\wedge\Big(\frac{N}{Ck}\|\theta-\hat\theta\|\Big).\]
Now consider the set $\{\theta:\alpha_{\mathcal{Q}|\mathbf{x}}(\theta) \leq\rho\}$. We know that for all $\theta$ in this set,
\[\Big(\frac{N}{2k}\|\theta-\hat\theta\|^2\Big)\wedge\Big(\frac{N}{Ck}\|\theta-\hat\theta\|\Big)\leq \rho\]
which implies
\[\|\theta-\hat\theta\| \leq \frac{Ck\rho}{N} \vee \sqrt{\frac{2k\rho}{N}}=:R.\]
\end{proof}

\bibliographystyle{plain}  
\bibliography{llikpenBib}
\end{document}